\documentclass[11pt]{article}
\usepackage{amsmath,amssymb,amsthm}

\newcommand{\ssa}{\ensuremath{_{\alpha}}}
\newcommand{\ssb}{\ensuremath{_{\beta}}}
\newcommand{\ssg}{\ensuremath{_{\gamma}}}
\newcommand{\ssd}{\ensuremath{_{\delta}}}
\newcommand{\sset}{\ensuremath{_{\eta}}}
\newcommand{\lla }{\ensuremath{\ell\ssa}}
\newcommand{\llb}{\ensuremath{\ell\ssb}}
\newcommand{\llg}{\ensuremath{\ell\ssg}}
\newcommand{\lld}{\ensuremath{\ell\ssd}}
\newcommand{\llet}{\ensuremath{\ell\sset}}
\newcommand{\mua}{\ensuremath{\mu\ssa}}
\newcommand{\mub}{\ensuremath{\mu\ssb}}
\newcommand{\mug}{\ensuremath{\mu\ssg}}
\newcommand{\mud}{\ensuremath{\mu\ssd}}

\newcommand{\Tbar}{\ensuremath{\overline{\mathcal{T}}}}
\newcommand{\Mbar}{\overline{\mathcal{M}}}

\newcommand{\grad}{\operatorname{grad}}
\newcommand{\inj}{\operatorname{inj}}
\newcommand{\spn}{\operatorname{span_{\mathbb C}}}
\newcommand{\msc}{\operatorname{msc}}

\newcommand{\caB}{\mathcal{B}}

\newcommand{\caH}{\mathcal{H}}

\newcommand{\caM}{\mathcal{M}}
\newcommand{\caP}{\mathcal{P}}

\newcommand{\caS}{\mathcal{S}}
\newcommand{\caT}{\mathcal{T}}

\newtheorem{definition}{Definition}
\newtheorem{lemma}[definition]{Lemma}
\newtheorem{theorem}[definition]{Theorem}
\newtheorem{corollary}[definition]{Corollary}
\newtheorem{proposition}[definition]{Proposition}

\newtheorem*{thrm}{Theorem}
\newtheorem*{coro}{Corollary}

\begin{document}

\title{Geodesic-length functions and the Weil-Petersson curvature tensor\footnote{2000 Mathematics Subject Classification Primary: 30F60, 32G15, 32Q05, 53B20.}}         
\author{Scott A. Wolpert\footnote{Partially supported by National Science Foundation grant DMS - 1005852.}}        
\date{October 2, 2011} 
\maketitle
\begin{abstract}
An expansion is developed for the Weil-Petersson Riemann curvature tensor in the thin region of the Teichm\"{u}ller and moduli spaces.  The tensor is evaluated on the gradients of geodesic-lengths for disjoint geodesics.  A precise lower bound for sectional curvature in terms of the surface systole is presented.  The curvature tensor expansion is applied to establish continuity properties at the frontier strata of the augmented Teichm\"{u}ller space.   The curvature tensor has the asymptotic product structure already observed for the metric and covariant derivative.  The product structure is combined with the earlier negative sectional curvature results to establish a classification of asymptotic flats.  Furthermore, tangent subspaces of more than half the dimension of Teichm\"{u}ller space contain sections with a definite amount of negative curvature.  Proofs combine estimates for uniformization group exponential-distance sums and potential theory bounds. 
\end{abstract}

\section{Introduction}
Let $\caT$ be the Teichm\"{u}ller space of marked genus $g$, $n$-punctured Riemann surfaces with hyperbolic metrics.   Associated to the hyperbolic metrics on Riemann surfaces are the Weil-Petersson (WP) K\"{a}hler metric and geodesic-length functions on $\caT$.  The WP metric is incomplete.   The metric completion is the augmented Teichm\"{u}ller space $\Tbar$ \cite{Abdegn, Bersdeg, Msext, Wlcbms}. The completion is a $CAT(0)$ metric space - a simply connected, complete metric space with non positive curvature geometry \cite{DW2,Wlcomp,Yam2}.  

In recent works, a collection of authors have applied curvature expansions to establish significant properties for WP geometry.  In a series of papers \cite{LSY1,LSY2,LSY6,LSY3,LSY7}, Liu-Sun-Yau have effectively used the negative of the Ricci form as a general K\"{a}hler comparison metric.  The authors show that, except for WP, the canonical metrics for $\caT$ are bi Lipschitz equivalent. In particular the canonical metrics are complete. They further show that the WP metric is Mumford good for the logarithmic polar cotangent bundle for the Deligne-Mumford compactification of the moduli space of Riemann surfaces; the Chern currents computed from the metric are closed and represent the Chern classes on the compactification.  They also show that the complete K\"{a}hler-Einstein metric is Mumford good and has bounded geometry.   A Gauss-Bonnet theorem for the metrics is established and they find that the logarithmic polar cotangent bundle is Mumford stable with respect to its first Chern class.   

Burns-Masur-Wilkinson show that the WP geodesic flow for the unit tangent bundle (defined almost everywhere) for the moduli space $\caM$ of Riemann surfaces is ergodic with finite, positive metric entropy \cite{BMW}. The authors follow the general Hopf approach in the form of Pesin theory and study Birkhoff averages along the leaves of the stable and unstable foliations of the flow.   In particular they use the Katok-Strelcyn result \cite{KaSt} on the existence and absolute continuity of stable and unstable manifolds for singular non uniformly hyperbolic systems.  The authors combine our expansions for the metric and curvature tensor evaluated on gradients of geodesic-length functions to show that for a trajectory segment, the first derivative of geodesic flow is bounded in terms of a reciprocal power of the distance of the trajectory segment to the boundary of $\caT\subset\Tbar$.  The authors use McMullen's Quasi Fuchsian Reciprocity \cite{McM} to similarly bound the second derivative of geodesic flow.  

Cavendish-Parlier consider the WP diameter of the moduli space $\caM$ and show for large genus that the ratio $diam(\caM)/\sqrt g $ is bounded above by a constant multiple of $\log g$ and below by a positive constant \cite{CaPa}.  The authors further find for genus fixed and a large number of punctures that the ratio $diam(\caM)/\sqrt n$ tends to a positive limit, independent of the genus.  To obtain a diameter upper bound, they refine Brock's quasi isometry of $(\caT,d_{WP})$ to the pants graph $\mathcal{PG}$ \cite{Brkwp}.  The refinement involves the bound \cite{Wlcomp} for distance to a stratum of $\Tbar$ in terms of geodesic-lengths and the diastole of a hyperbolic surface.  They also refine the pants graph $\mathcal{PG}$ by adding diagonals for multi dimensional cubes.  They then solve the asymptotic combinatorial problem of finding an upper bound for the diameter of the quotient of $\mathcal{PG}$ by the mapping class group.  To obtain a diameter lower bound, they combine geodesic convexity and the product structure of strata.  Cavendish-Parlier show that the refined pants graph models WP geometry on a scale comparable to the diameter.    
 
Our purpose is to give an expansion for the WP Riemann curvature tensor in the {\em thin region} of the Teichm\"{u}ller and moduli spaces.  Understanding is facilitated by introducing an explicit complex frame for the tangent bundle.  We consider the curvature tensor evaluated on the gradients $\lambda\ssa=\grad\lla^{1/2}$ of roots of geodesic-lengths $\lla$.  The general approach was used in \cite{Wlbhv} to obtain expansions for the metric and in \cite{Wlbhv, Wlext} to obtain expansions for the Levi-Civita connection $D$.  The expansion for the WP Hermitian form (expansion (\ref{wpherm}) in Section \ref{tensorsec}) is
\begin{equation*}
\langle\langle\lambda\ssa,\lambda\ssb\rangle\rangle\,=\,\frac{\delta_{\alpha\beta}}{4\pi}\,+\,O(\lla^{3/2}\llb^{3/2})
\end{equation*}
for $\alpha,\beta$ simple closed geodesics, coinciding or disjoint, where given $c_0>0$, the remainder term constant is uniform for $\lla,\llb\le c_0$.  We combine a method for estimating uniformization group sums and potential theory estimates to obtain expansions for the quantities for the curvature tensor.  In general our expansions are uniform in surface dependence and independent of topological type. 

The curvature operator is the commutator of covariant differentiation
\[
R(U,V)W\,=D_UD_VW\,-\,D_VD_UW\,-D_{[U,V]}W.
\]
The Riemann curvature tensor $\langle R(U,V)W,X\rangle$ is defined on tangent spaces. Tangent spaces $T$ can be complexified by tensoring with the complex numbers.  Bochner discovered general symmetries of the curvature tensor for a K\"{a}hler metric \cite{Boch}.  In particular the complexified tensor has a block form relative to the tangent space decomposition $\mathbb C\otimes T=T^{1,0}\oplus T^{0,1}$ into {\em holomorphic} and {\em non holomorphic type}. The complexified tensor is determined by its  $T^{1,0}\times\overline{T^{1,0}}\times T^{1,0}\times\overline{T^{1,0}}$ evaluation. We consider the complexified tensor.  For a pants decomposition $\caP$, a maximal collection of disjoint simple closed geodesics, we consider the gradients $\{\lambda\ssa\}_{\alpha\in\caP}$, $\lambda\ssa=\grad\lla^{1/2}$.  The pants root-length gradients provide a global frame for the tangent bundle $T^{1,0}\caT$ \cite{WlFN}.   We establish the following results in Theorem \ref{mainthm} and Corollary \ref{wpseccurv}. 
\begin{thrm}
The WP curvature tensor evaluation for the root-length gradient for a simple closed geodesic satisfies
\[
R(\lambda\ssa,\lambda\ssa,\lambda\ssa,\lambda\ssa)\,=\,\frac{3}{16\pi^3\lla}\,+\,O(\lla),
\]
and
\[
R(\lambda\ssa,\lambda\ssa,\ ,\ )\,=\,\frac{3|\langle\langle\lambda\ssa,\ \rangle\rangle|^2}{\pi\lla}\,+\,O(\lla\|\ \|_{WP}^2).
\]
For simple closed geodesics, disjoint or coinciding, with at most pairs coinciding, the curvature evaluation $R(\lambda\ssa,\lambda\ssb,\lambda\ssg,\lambda\ssd)$ is bounded as $O((\lla\llb\llg\lld)^{1/2})$.  
\newline The root-length holomorphic sectional curvature satisfies 
\[
K(\lambda\ssa)\,=\,\frac{-3}{\pi\lla}\,+\,O(\lla)
\]
and given $\epsilon >0$, there is a positive constant $c_{g,n,\epsilon}$  such that for the systole $\Lambda(R)$ of a surface, the sectional curvatures at $R\in\caT$ are bounded below by 
\[
\frac{-3-\epsilon}{\pi\Lambda(R)}\,-\, c_{g,n,\epsilon}.
\]
For $\alpha,\beta$ disjoint with $\lla,\llb\le\epsilon$, the sections spanned by 
$(J)\lambda\ssa,(J)\lambda\ssb$ have curvature bounded as $O(\epsilon^4)$.  
For $c_0$ positive, the remainder term constants are uniform in the surface $R$ and independent of the topological type for $\lla,\llb,\llg,\lld\le c_0$.  
\end{thrm}  
\noindent An open question is to find a precise upper bound for sectional curvatures in terms of the surface systole. 

The augmented Teichm\"{u}ller space $\Tbar$ is described in terms of Fenchel-Nielsen coordinates $(\lla,\vartheta\ssa)_{\alpha\in\caP}$, $\caP$ a pants decomposition, and the Chabauty topology for $PSL(2;\mathbb R)$ representations.  The partial compactification is introduced by extending the range of the Fenchel-Nielsen parameters.   For a length coordinate $\lla$ equal to zero, the angle $\vartheta\ssa$ is not defined and in place of a geodesic $\alpha$ on the surface, there appears a pair of cusps.  The new points describe unions of hyperbolic surfaces, {\em noded Riemann surfaces}, with {\em components} and formal pairings of cusps, {\em nodes}.   The genus of a union of hyperbolic surfaces is defined by the relation $Total\ area=2\pi(2g-2+n)$.   For a disjoint collection of geodesics $\sigma\subset\caP$, the $\sigma$-null stratum is the locus of noded Riemann surfaces $\caT(\sigma)=\{R\mbox{ a union }\mid\lla(R)=0 \mbox{ iff } \alpha\in\sigma\}$.   The partial compactification is $\Tbar =\caT\cup_{\sigma}\caT(\sigma)$ for the union over all homotopically distinct disjoint collections of geodesics.  Neighborhood bases for points of $\caT(\sigma)\subset\Tbar$  are specified by the condition that for each pants decomposition $\caP$, $\sigma\subset\mathcal P$, the projection
$((\ell_{\beta},\vartheta_{\beta}),\ell_{\alpha}):
\mathcal{T}\cup\mathcal{T}(\sigma)\rightarrow\prod_{\beta\in\caP-\sigma}(\mathbb{R}_+\times\mathbb{R})\times\prod_{\alpha\in\sigma}(\mathbb{R}_{\ge 0})$  is continuous. 

The formulas for the metric \cite{DW2,Msext,Wlbhv}, covariant derivative \cite{Wlbhv,Wlext} and curvature tensor display the asymptotic product structure
\[
\prod_{\alpha\in\sigma}\,\spn\{\lambda\ssa\}\ \times\prod_{R^{\natural}\in\,
\operatorname{parts}^{\sharp}(\sigma)}T^{1,0}\caT(R^{\natural}),
\]
 for an extension of the tangent bundle over $\caT(\sigma)$, where \emph{parts}$^{\sharp}(\sigma)$ are the components of the $\sigma$-complement, that are not thrice-punctured spheres, for the surfaces represented in $\caT(\sigma)$.  Evaluations involving more than a single factor of the product tend to zero approaching $\caT(\sigma)$. Evaluations for a single factor tend to evaluations for either the standard metric for opening a node or for a lower dimensional Teichm\"{u}ller space.  The structure is formal since $\Tbar$ is not a complex manifold and the corresponding extension of the vector bundle of holomorphic quadratic differentials over $\Mbar$ is not the cotangent bundle, but the logarithmic polar cotangent bundle \cite{HMbook}. Nevertheless the product structure applies for limits of the metric and curvature tensor.  The individual product factors have strictly negative sectional curvature.   The $\lambda\ssa$-section is holomorphic with curvature described in the above Theorem.  The earlier general result \cite{Trcurv,Wlchern} establishes negative curvature for the Teichm\"{u}ller spaces $\caT(R^{\natural})$.  Recall for a product of negatively curved manifolds, a zero curvature tangent section has at most one $\mathbb R$-dimensional projection into the tangent space of each factor.  
We establish the counterpart for WP.  By considering $\mathbb C$-sums of the indeterminates $\lambda\ssg$, $\gamma\in\caP$, a germ $\mathcal V$  is defined for an extension over 
$\caT(\sigma)$ of the tangent bundle of $\caT$.  A formal product structure for $\mathcal V$ is defined by the present considerations.  We establish in Theorem \ref{curvcont} that, except for the diagonal evaluation, curvature tensor evaluations are continuous at $\caT(\sigma)$.  Continuity provides that curvatures near $\caT(\sigma)$ can be understood in terms of evaluations at $\caT(\sigma)$.   An application is the classification in Corollary \ref{flats}.  

\begin{coro}\textup{Classification of asymptotic flats.}  Let $\mathcal S$ be a $\mathbb R$-subspace of the fiber of $\mathcal V$ over a point of $\caT(\sigma)$.  The subspace $\mathcal S$ is a limit of a sequence of tangent multisections over points of $\caT$ with all sectional curvatures tending to zero if and only if the projections of $\mathcal S$ onto the factors of the product structure are at most one $\mathbb R$-dimensional.  The maximal dimension for $\mathcal S$ is $|\sigma|+|\operatorname{parts}^{\sharp}(\sigma)|\le\dim_{\mathbb C}\caT$.
\end{coro}

The gradient $\grad\lla$ describes infinitesimal pinching of the length of $\alpha$ and the corresponding infinitesimal Fenchel-Nielsen twist is $t\ssa=\frac{i}{2}\grad\lla$ \cite{WlFN,Wlcusps}.   The flats classification includes the Huang result \cite{Zh2} that independent pinchings and twists describe asymptotically flat tangent sections.  Flat subspaces are an important consideration in global geometry.  Gromov defined the rank of a metric space to be the maximal dimension of a quasi isometric embedding of a Euclidean space.  Brock-Farb \cite{BF} show that the rank of $\caT$ in the sense of Gromov is at least  
$\lfloor(1+\dim_{\mathbb C}\caT)/2\rfloor$. The number $\mathbf m=\lfloor(1+\dim_{\mathbb C}\caT)/2\rfloor$ is the maximum number of factors for the second product of the asymptotic product structure.  Behrstock-Minsky \cite{BMi} show that the Gromov rank is exactly $\mathbf m$.  In \cite[Section 6]{Wlcomp} we find that a locally Euclidean subspace of $\Tbar$ has dimension at most $\mathbf m$. The above result provides additional information about rank and asymptotic flats. The considerations show that beyond asymptotic flats there is a negative upper bound for sectional curvature, Corollary \ref{beyflat}.

\begin{coro}
There exists a negative constant $c_{g,n}$ such that a subspace $\caS$ of a tangent space of $\caT$ with $\dim_{\mathbb R}\caS > \dim_{\mathbb C}\caT$ contains a section with sectional curvature at most $c_{g,n}$.  
\end{coro}
 
Central to our considerations is the space of holomorphic quadratic differentials $Q(R)$ and the elements representing the differentials of geodesic-lengths.  In Section \ref{secgrad} we discuss that for small geodesic-length, the differential for a finite area hyperbolic surface is closely approximated by the differential for the cyclic cover corresponding to the geodesic.  The approximation is basic to our analysis.  Comparing the WP and Teichm\"{u}ller metrics and developing expansions for the WP metric involves a comparison of norms for $Q(R)$.  For $ds^2$ the complete hyperbolic metric on $R$ and $\phi\in Q(R)$, let $\|\varphi\|_{\infty}=\sup |\varphi(ds^2)^{-1}|$ and $\|\varphi\|_2$ be the WP norm.  From the analysis for differentials of geodesic-length, we establish an asymptotic norm comparison in Corollary \ref{pantsbd}. 
\begin{coro}\textup{Comparison of norms.}
Given $\epsilon>0$, there is a positive value $\Lambda_0$, such that for the surface systole $\Lambda(R)\le\Lambda_0$, the maximal ratio of $L^{\infty}$ and $L^2$ norms for $Q(R)$ satisfies 
\[
(1-\epsilon)\Big(\frac{2}{\pi\Lambda(R)}\Big)^{1/2}\,\le\,\max_{\varphi\in Q(R)}\frac{\|\varphi\|_{\infty}}{\|\varphi\|_2}\,\le\,(1+\epsilon)\Big(\frac{2}{\pi\Lambda(R)}\Big)^{1/2}.
\]
\end{coro}     

\section{Preliminaries}\label{prelim}
We follow the exposition of \cite{Wlext}.  Points of the Teichm\"{u}ller space $\caT$ are equivalence classes $\{(R,ds^2,f)\}$ of marked genus $g$, $n$-punctured Riemann surfaces with complete hyperbolic metrics and reference homeomorphisms $f:F\rightarrow R$ from a base surface $F$.  Triples are equivalent provided there is a conformal isomorphism of Riemann surfaces homotopic to the composition of reference homeomorphisms.  Basic invariants of a hyperbolic metric are the lengths of the unique closed geodesic representatives of the non peripheral free homotopy classes.  For a non peripheral free homotopy class $[\alpha]$ on $F$, the length of the unique geodesic representative for $f(\alpha)$ is the value of the geodesic-length $\ell_{\alpha}$ at the marked surface.  For $R$ with uniformization representation $f_*:\pi_1(F)\rightarrow\Gamma \subset PSL(2;\mathbb R)$ and $\alpha$ corresponding to the conjugacy class of an element $A$ then $\cosh \ell_{\alpha}/2= tr A/2$.  Collections of geodesic-lengths provide local $\mathbb R$-coordinates for $\caT$, \cite{Busbook,ImTan,WlFN}. 

From Kodaira-Spencer deformation theory the infinitesimal deformations of a surface $R$ are represented by the Beltrami differentials $\caH (R)$ harmonic with respect to the hyperbolic metric, \cite{Ahsome}. Also the cotangent space of $\caT$ at $R$ is $Q(R)$ the space of holomorphic quadratic differentials with at most simple poles at the punctures of $R$.  The holomorphic tangent-cotangent pairing is
\[
(\mu,\varphi)=\int_R \mu \varphi
\]
for $\mu\in\caH(R)$ and $\varphi\in Q(R)$.  Elements of $\caH (R)$ are symmetric tensors given as $\overline\varphi(ds^2)^{-1}$ for $\varphi\in Q(R)$ and $ds^2$ the hyperbolic metric.  The Weil-Petersson (WP) Hermitian metric and cometric pairings are given as
\[
\langle\mu,\nu\rangle =\int_R\mu\overline\nu\, dA\quad\mbox{and}
\quad\langle\varphi,\psi\rangle =\int_R\varphi\overline\psi (ds^2)^{-1}
\]
for $\mu,\nu\in\caH (R)$ and $\varphi, \psi\in Q(R)$ and $dA$ the hyperbolic area element.  The WP Riemannian metric is $\Re\langle\ ,\ \rangle$.  The metric is K\"{a}hler, non complete, with non pinched negative sectional curvature and determines a $CAT(0)$ geometry for Teichm\"{u}ller space, see \cite{Ahsome, Zh2, Trmbook, Wlcomp, Wlcbms} for references and background. 
     
A Riemann surface with hyperbolic metric can be considered as the union of a $thick$ region where the injectivity radius is bounded below by a positive constant and a complementary $thin$ region.  The totality of all $thick$ regions of Riemann surfaces of a given topological type forms a compact set of metric spaces in the Gromov-Hausdorff topology.  A $thin$ region is a disjoint union of collar and cusp regions.  We describe basic properties of collar and cusp regions including bounds for the injectivity radius and separation of simple geodesics. 

We follow Buser's presentation \cite[Chap. 4]{Busbook}.   For a geodesic $\alpha$ of length $\ell_{\alpha}$ on a Riemann surface, the extended collar about the geodesic is $\hat c(\alpha)=\{d(p,\alpha)\le w(\alpha)\}$ for the width $w(\alpha)$, $\sinh w(\alpha)\sinh \ell_{\alpha}/2=1$.  The width is given as $w(\alpha)=\log 4/\ell_{\alpha}+O(\ell_{\alpha}^2)$ for $\ell_{\alpha}$ small.  For $\mathbb H$ the upper half plane with hyperbolic distance $d(\ ,\ )$, an extended collar is covered by the region  $\{d(z,i\mathbb R^+)\le w(\alpha)\}\subset \mathbb H$ with deck transformations generated by $z\rightarrow e^{\ell_{\alpha}}z$.  The quotient $\{d(z,i\mathbb R^+)\le w(\alpha)\}\slash \langle z\rightarrow e^{\ell_{\alpha}}z\bigr >$ embeds into the Riemann surface. For $z$ in $\mathbb H$, the  region is approximately $\{\ell_{\alpha}/2\le \arg z\le \pi-\ell_{\alpha}/2\}$.   
An extended cusp region is covered by the region $\{\Im z\ge 1/2\}\subset\mathbb H$ with deck transformations generated by $z\rightarrow z+1$.  The quotient $\{\Im z\ge 1/2\}\slash \langle z\rightarrow z+1\rangle $ embeds into the Riemann surface.  To ensure that uniform bands around boundaries embed into the Riemann surface, we will use {\em collars} $c(\alpha)$ defined by covering regions $\{\ell_{\alpha}\le \arg z\le \pi-\ell_{\alpha}\}$ and {\em cusp regions} defined by covering regions $\{\Im z\ge 1\}$.  Collars are contained in extended collars and cusp regions are contained in extended cusp regions.  The boundary of a collar $c(\alpha)$ for $\ell_{\alpha}$ bounded and boundary of a cusp region have length approximately $1$.

\begin{theorem}
\label{collars}
For a Riemann surface of genus $g$ with $n$ punctures, given pairwise disjoint simple closed geodesics $\alpha_1,\dots,\alpha_m$, there exist simple closed geodesics $\alpha_{m+1},\dots,\alpha_{3g-3+n}$ such that $\alpha_1,\dots,\alpha_{3g-3+n}$ are pairwise disjoint.  The extended collars $\hat c(\alpha_j)$ about $\alpha_j$, $1\le j\le 3g-3+n$, and the extended cusp regions are mutually pairwise disjoint.
\end{theorem}

On $thin$ the injectivity radius is bounded above and below in terms of the distance into a collar or cusp region.  For a point $p$ of a collar or cusp region, write $\inj(p)$ for the injectivity radius and $\delta(p)$ for the distance to the boundary of the collar or cusp region.  The injectivity radius is bounded as follows, \cite[II, Lemma 2.1]{Wlspeclim}, \cite{Wlcbms}.

\begin{lemma}\textup{Quantitative Collar and Cusp Lemma}
\label{enhcollar}
The product $inj(p)\,e^{\delta(p)}$ of injectivity radius and exponential distance to the boundary is bounded above and below by uniform positive constants.  
\end{lemma}

\noindent As a general point, we note that the standard consideration for simple closed geodesics and cusp regions generalizes as follows, \cite[Lemma 2.3]{Wlbhv}.

\begin{lemma}
\label{separ}
A simple closed geodesic is either disjoint from the thin region or is the core of an included collar or crosses an included collar.
\end{lemma} 

As a final point, we note for elements of $\mathcal H(R)$ that on a cusp region, magnitude is uniformly bounded in terms of the maximum on the boundary.  A variable for a cusp region is $w=e^{2\pi iz},\, |w|\le e^{-\pi},$ for $z$ in $\mathbb H$ and the cusp represented as above.  
\begin{lemma}\label{cuspbd}
For $w$ the given cusp region variable, a harmonic Beltrami differential is bounded as 
$|\mu|\le|w|((\log|w|)/\pi)^2\max_{|w|=e^{-\pi}}|\mu|$.
\end{lemma}
\begin{proof}
The hyperbolic metric for a cusp region is $ds^2=(|dw|/|w|\log |w|)^2$.  In a cusp region, a harmonic Beltrami differential $\mu$ is given as  $\overline{\mu}=f(w)(dw/w)^2(|w|\log |w|/|dw|)^2$ for $f(w)$ holomorphic and vanishing at the origin.  Apply the Schwarz Lemma for the disc $|w|\le e^{-\pi}$, to find the inequality $|f|\le e^{\pi}|w|\max_{|w|=e^{-\pi}}|f|$ or equivalently the inequality $|\mu|\le e^{\pi}|w|(\log|w|)^2\max_{|w|=e^{-\pi}}|f|$.  Finally  note the equality $|f|=|\mu|/\pi^2$ on the boundary.  
\end{proof}

\section{Gradients of geodesic-length functions}\label{secgrad}
Geodesic-length functions are a fundamental tool of Teichm\"{u}ller theory.  A point of the Teichm\"{u}ller space $\caT(F)$ represents a Riemann surface $R$ and an isomorphism of $\pi_1(F)$ to the deck transformation group $\pi_1(R)$.  For a free homotopy class $[\alpha]$ of a non trivial, non peripheral closed curve on the reference surface $F$, the geodesic-length $\lla(R)$ is the length of the unique closed geodesic in the corresponding free homotopy class on $R$.   For the uniformization group $\Gamma$ conjugated in $PSL(2;\mathbb R)$ and a representative $\alpha$ having the imaginary axis $\mathcal I$ as a component of its lift with cyclic stabilizer $\Gamma\ssa\subset\Gamma$, we consider a coset sum.
\begin{definition}
Associated to the geodesic $\alpha$ is the coset sum
\[
\Theta\ssa\,=\,\frac{2}{\pi}\sum_{A\in\Gamma\ssa\backslash\Gamma}A^*\big(\frac{dz}{z}\big)^2\,\in\,Q(R)
\]
and the harmonic Beltrami differential $\mua\,=\,\overline{\Theta\ssa}(ds^2)^{-1}\in\mathcal H(R)$.
\end{definition}
The differential $(dz/z)^2$ is invariant under the subgroup of $PSL(2;\mathbb R)$ stabilizing $\mathcal I$ and the summands are independent of choices of coset representatives.  We consider bounds for $\mua$.  For $z=re^{i\theta}\in\mathbb H$, the distance to the imaginary axis is $d(\mathcal I,z)=\log(\csc \theta+|\cot\theta|)$ and consequently the inverse square exponential-distance $e^{-2d(\mathcal I,z)}$ is comparable to the function $\sin^2\theta$.   Our considerations involve the elementary Beltrami differential 
\[
\omega\,=\,\overline{\big(\frac{dz}{z}\big)^2}(ds^2)^{-1}.  
\]
The elementary differential satisfies $|\omega|=\sin^2\theta\le4e^{-2d(\tilde\alpha,z)}$ for $\tilde\alpha$ the imaginary axis.  We recall the basic estimates for $\Theta_{\alpha}$ and $\mu_{\alpha}$, given in \cite[II, Lemmas 2.1 and 2.2]{Wlspeclim}, \cite[Lemma 4.3]{Wlext}. 
      
\begin{proposition}\label{thetprop}
The harmonic Beltrami differential $\mua(p)$ is bounded as $O(\inj(p)^{-1}\lla e^{-d(\alpha,p)})$ for $\inj(p)$ the injectivity radius at $p$ and $d(\alpha,p)$ the distance of $p$ to the geodesic.   On the $c(\alpha)$ collar complement, $\mua$ is bounded as $O(\lla^2)$.   For $c_0$ positive, the remainder term constant is uniform in the surface $R$ and independent of the topological type for $\lla\le c_0$.
\end{proposition}
\begin{proof}We use the distant sum estimate \cite[Chapter 8]{Wlcbms}.  By a mean value estimate the individual terms of the sum $\Theta_{\alpha}$ are bounded by integrals over corresponding balls. The sum of integrals over balls is bounded by the integral over a region containing the balls.  We elaborate.  Harmonic Beltrami differentials and the elementary differential satisfy a mean value estimate on $\mathbb H$: given $\epsilon >0$, $|\omega(p)|\le c_{\epsilon}\int_{B(p;\epsilon)}|\omega|dA$ for the hyperbolic area element.  The orbit $Ap\in \mathbb H,\ A\in\Gamma\ssa\backslash\Gamma$, is necessarily contained in the exterior sector $\{z\mid d(z,\mathcal I)\ge d(p,\alpha)\}\subset\mathbb H$.  The sum is consequently bounded as follows
\[
\sum_{A\in\Gamma\ssa\backslash\Gamma}|A^*\omega|\,\le\,\sum_{A\in\Gamma\ssa\backslash\Gamma}c_{\epsilon}\int_{B(p;\epsilon)}|\omega|dA\,\le\,c_{\epsilon}\inj(p)^{-1}\int_{\cup_{A\in\Gamma\ssa\backslash\Gamma}B(Ap;\epsilon)}|\omega|dA
\]
where the reciprocal injectivity radius bounds the count of overlapping balls in $\mathbb H$.   From the equality $|\omega|=\sin^2\theta$, the final integrand in polar coordinates is $drd\theta/r$.   The region of integration is represented in the half annulus $\mathcal A\ssa=\{1\le|z|<e^{\lla}\}$, since the union is for $\Gamma\ssa$-cosets.  The region of integration is  contained in the exterior sector $\{z\mid d(z,\mathcal I)\ge d(p,\alpha)\}$.  The distance inequality is equivalent to the inequality $\theta,\pi-\theta\le\theta(p)$, where $e^{d(p,\alpha)}=\csc\theta(p)\,+\,\cot\theta(p)$.  The integral of $drd\theta/r$ over $\theta,\pi-\theta\le \theta(p), 1\le r<e^{\ell_{\alpha}}$ is bounded by a constant multiple of $\lla e^{-d(\alpha,p)}$, giving the first conclusion.  The injectivity radius is small on cusp regions and  collars of a hyperbolic surface.  For $p$ in a cusp region or a collar other than $c(\alpha)$, the initial segment of a geodesic from $p$ to $\alpha$ exits the cusp region or second collar and the final segment of the geodesic crosses the $\alpha$ collar to connect to $\alpha$.  By the Quantitative Collar and Cusp Lemma, the product $\inj(p)^{-1} e^{-d(\alpha,p)}$ is bounded by $\lla$, providing the second conclusion.      
\end{proof}  

Our considerations involve carefully analyzing holomorphic quadratic differentials on collars.   The Laurent series expansion provides an approach.  For the closed geodesic $\alpha$ having the imaginary axis as a component of its lift, the extended collar $\hat c(\alpha)$ is covered 
by $\{\lla/2\le\arg z\le\pi-\lla/2\}$.  The following compares to results on decompositions of differentials on annuli developed in \cite[Section 4]{HSS}.     
\begin{proposition}\label{CSest}  Given the holomorphic quadratic differential $\phi\in Q(R)$,  the coordinate $z=re^{i\theta}\in\mathbb H$, and the elementary Beltrami differential $\omega$, then on the sector $\lla\le\theta\le\pi-\lla$,
\[
\overline{\phi}(ds^2)^{-1}=a(\alpha)\omega +O\big((e^{-2\pi\theta/\lla}+e^{2\pi(\theta-\pi)/\lla})\ell^{-2}_{\alpha}\sin^2\theta\,M\big)
\]
for a coefficient $a(\alpha)$ and $M=\max_{|z|=\lla/2,\pi-\lla/2}|\overline{\phi}(ds^2)^{-1}|$.  For $c_0$ positive, the remainder term is uniform in the surface  $R$ for $\lla\le c_0$.
\end{proposition}
\begin{proof}
The uniformizing map $w=exp(2\pi i\log z/\lla)$ represents $\mathbb H\slash \Gamma\ssa$ as the annulus $\{e^{-2\pi^2/\lla}<|w|<1\}$.  The quadratic differential $\phi$ has a Laurent series expansion
\[
\phi\,=\,\big(f_+(w)+a_0+f_-(w)\big)\big(\frac{1}{w^2}\big)=\sum_{n=-\infty}^{\infty}a_nw^{n-2}.
\] 
The function $f_+(w)+a_0$ is holomorphic in $|w|<1$ and is given by the Cauchy Integral of $w^2\phi$ on a circle $|w|=e^{-\pi}$ ($\arg z=-\lla/2$).  The main coefficient $a_0$ is the average of $w^2\phi$ on the circle.   The function $f_-(w)$ is holomorphic in $|w|>e^{-2\pi^2/\lla}$ and is given by the Cauchy Integral of $w^2\phi$ on the circle $|w|=e^{\pi-2\pi^2/\lla}$    

By the elementary estimate for the Cauchy Integral and the Schwarz Lemma for $|w|\le e^{-2\pi}$, it follows that $|f_+(w)|\le C|w|\, \max_{|w|=e^{-\pi}}|w^2\phi|$ for $|w|\le e^{-2\pi}$ and a universal constant.   The factor $|w|^2$ can be absorbed into the constant.  Next we consider the hyperbolic metric on the annulus.  The metric for the annulus $\mathbb H\slash\Gamma\ssa$ is
$ds^2\,=\,((\lla/2\pi)\csc (\lla/2\pi\log |w|)\,|dw/w|)^2$ which is approximately $(e^{\pi}|dw|/\pi)^2$ for $|w|=e^{-\pi}$. We combine the considerations for $f_+$ and the hyperbolic metric to write the final estimate $|f_+(w)|\le C|w|\,\max_{|w|=e^{-\pi}}|\overline{\phi}(ds^2)^{-1}|$ for $|w|\le e^{-2\pi}$.  

The estimate for $f_-$ follows by a symmetry consideration. The inversion $w\rightarrow 1/we^{2\pi^2/\lla}$ interchanges the boundaries of the annulus, interchanges the functions $f_+,f_-$ and replaces the angle $\theta=\arg z$ with the angle $\pi-\theta$.  The final estimate is $|f_-(w)|\le C |we^{2\pi^2/\lla}|\,\max_{|w|=e^{\pi-2\pi^2/\lla}}|\overline{\phi}(ds^2)^{-1}|$ for $|w|\ge e^{2\pi-2\pi^2/\lla}$ follows.  Each inequality is multiplied by the absolute value of the elementary differential $(\overline{dw/w})^2(ds^2)^{-1}=-(2\pi/\ell_{\alpha})^2\omega$, where $|\omega|=\sin^2\theta$, to obtain the desired intrinsic statement.    
\end{proof}

We collect prior results to provide a description of pairings with the geodesic-length differentials $d\lla$ and gradients $\grad\lla$ \cite{Gardtheta,Rier,Wlbhv}.  An exposition of prior results is provided in Chapters 3 and 8 of \cite{Wlcbms}.  
  
\begin{theorem}\label{gradpair}\textup{Geodesic-length gradients and their pairings.}  For a closed geodesic $\alpha$ and harmonic Beltrami differential $\mu\in\mathcal H(R)$, then
\[
\Re (\mu,d\lla)\,=\,\Re\int_R\mu\Theta\ssa\,=\,\Re\, \langle\mu,\mua\rangle.
\]
For the geodesic $\alpha$ with the standard representation in $\mathbb H$ and a holomorphic quadratic differential $\phi\in Q(R)$ with Laurent type expansion
$\phi=\sum a_n exp(2\pi in\log z/\lla)\,(dz/z)^2$ on $\mathbb H$, the main coefficient $a_0$ is given by $(\mua,\phi)=\langle\mua,\overline\phi(ds^2)^{-1}\rangle = a_0(\phi)\lla$.  The WP Hermitian pairing of geodesic-length gradients for simple geodesics $\alpha,\beta$, coinciding or disjoint, is positive real-valued and satisfies 
\[
Ce^{-2\ell_{\widehat{\alpha\beta}}} \le\langle \grad\lla,\grad\llb\rangle-\frac{2}{\pi}\lla\delta_{\alpha\beta}\le C'\lla^2\llb^2
\] 
for the Dirac delta $\delta_*$, and $\ell_{\widehat{\alpha\beta}}$ the length of the shortest non trivial geodesic segment connecting $\alpha$ to $\beta$.  For $c_0$ positive, the positive constants $C,C'$ are uniform in the surface  $R$ and independent of the topological type for $\lla,\llb\le c_0$.
\end{theorem}
\begin{proof}We only sketch the considerations.  The statements about the differential and gradient are Gardiner's formulas \cite{Gardtheta}.  The remaining statements regard gradient pairings.  The Fenchel-Nielsen twist about $\alpha$ is represented by the harmonic Beltrami differential $i/2\grad\lla$.  The twist-length cosine formula provides that the pairing is real \cite{Wlcbms}.  
The Riera formula expresses the pairing of gradients as a positive infinite sum $\sum_{\Gamma\ssa\backslash\Gamma\slash\Gamma\ssb}e^{-2d}$, for the square inverse exponential-distance between components of the lifts of $\alpha$ and $\beta$ modulo the action of $\Gamma$ \cite{Rier}.  The expansion for the pairing is obtained by applying a form of the distant sum estimate \cite{Wlbhv}.  The lower bound follows by considering the first term of the sum. 
\end{proof}

\begin{corollary}\textup{Expansions of geodesic-length gradients on collars.}\label{gradexp}  Given simple geodesics $\alpha,\beta$, on a fundamental domain containing the collar $c(\alpha)$ with the representation $\{\lla\le \arg z=\theta\le \pi-\lla\}$ in $\mathbb H$, then
\begin{equation}\label{grada}
\grad\lla\,=\,a\ssa(\alpha)\omega\,+\,O((e^{-2\pi\theta/\lla}+e^{2\pi(\pi-\theta)/\lla})\sin^2\theta)
\end{equation} for the main coefficient $a\ssa(\alpha)=\frac{2}{\pi}+O(\lla^3)$, and for $\beta$ 
disjoint from $\alpha$
\begin{equation}\label{gradab}
\grad\llb\,=\,a\ssb(\alpha)\omega\,+\,O((\llb/\lla)^2(e^{-2\pi\theta/\lla}+e^{2\pi(\pi-\theta)/\lla})\sin^2\theta)
\end{equation}
for the main coefficient $a\ssb(\alpha)=O(\lla\llb^2)$.  The main coefficients $a\ssb(\alpha)=\langle\grad\llb,\grad\lla\rangle/\lla$ are positive real-valued.   For $c_0$ positive, the remainder term is uniform in the surface $R$ and independent of the topological type for $\lla,\llb\le c_0$.
\end{corollary}
\begin{proof}  
The expansions are obtained by combining Proposition \ref{CSest}, Theorem \ref{gradpair} and noting that $\sin\theta/\lla$ has approximately unit magnitude on the collar boundary.  
\end{proof}

We introduce the square roots of geodesic-lengths.
\begin{definition}
The geodesic root-length gradient is $\lambda\ssa=\grad\lla^{1/2}$.
\end{definition}
\noindent The Theorem \ref{gradpair}  expansion for the WP Riemannian inner product becomes
\begin{equation*}\label{rgradprod}
\langle\lambda\ssa,\lambda\ssb\rangle\,=\,\frac{\delta_{\alpha\beta}}{2\pi}\,+\,O(\lla^{3/2}\llb^{3/2}).
\end{equation*} 
There is an interpretation of $\grad\lla^{1/2}=\lla^{-1/2}\mua/2$ as an indicator function (differential) for the collar $c(\alpha)$.  The root-length gradient $(2\pi)^{1/2}\lambda\ssa$ has approximately unit norm and as observed in Proposition \ref{CSest}, the gradient is $O(\lla^{3/2})$ on the collar complement. From formula (\ref{grada}) on the collar $c(\alpha)$, the gradient $\grad\lla^{1/2}=\lla^{-1/2}a\ssa(\alpha)\omega/2\,+\,O(\lla^{3/2})$ is closely approximated by its main term; by Corollary \ref{gradexp} the remainder is exponentially small on the core of the collar.  Further for a collection of disjoint simple closed geodesics, since collars are disjoint, the root-length gradients approximately have disjoint supports.  These observations are used in the proof of Corollary \ref{pantsbd} below.  

A pants decomposition $\caP$ is a maximal collection $\{\alpha_1,\dots,\alpha_{3g-3+n}\}$ of disjoint simple closed geodesics.  In \cite[Theorem 3.7]{WlFN}, it is shown for a pants decomposition $\caP$ that the gradients $\{\grad\ell_{\alpha}\}_{\alpha\in\caP}$ provide a global $\mathbb C$-frame for the vector bundle of harmonic Beltrami differentials over $\caT$.  Equivalently, the gradients $\{\grad\ell_{\alpha}^{1/2}\}_{\alpha\in\caP}$ and the complex differentials $\{\partial\ell_{\alpha}^{1/2}\}_{\alpha\in\caP}$ provide global $\mathbb C$-frames for their  respective vector bundles over $\caT$.  Hatcher and Thurston observed that there are only a finite number of pants decompositions modulo the action of the mapping class group $Mod$.  To ensure uniform remainder terms in expansions, we consider bounded pants decompositions.  Bers found that there is a positive constant $b_{g,n}$ such that Teichm\"{u}ller space is covered by the bounded pants decomposition regions $\mathcal B(\mathcal P)=\{\ell_{\alpha_j}<b_{g,n},\,\alpha_j\in\mathcal P\}$, called Bers regions \cite{Busbook}.  We will give expansions in terms of Bers regions. 

A hyperbolic metric is described by its set of geodesic-lengths.  Basic behavior of the WP metric can be understood by having a model for the metric in terms of variations geodesic-lengths.   Explicit models are presented in Theorem 4.3 and Corollaries 4.4 and 4.5 of \cite{Wlbhv}.  We recall the basic comparison below.   Comparing the WP and Teichm\"{u}ller metrics, as well as developing expansions for the WP metric involves comparison of the $L^{\infty}$, $L^1$ and $L^2$ norms for harmonic Beltrami differentials or equivalently for holomorphic quadratic differentials.   For $\mu\in\mathcal H(R)$ then $\mu=\overline\varphi(ds^2)^{-1}$ and the mapping $\mu\mapsto\varphi$ of $\mathcal H(R)$ to $Q(R)$ is an isometry for each norm.  
An asymptotic decomposition and analysis of concentration of holomorphic quadratic differentials in terms of the $L^1$ norm is given in \cite[Sections 4 and 5]{HSS}.  The authors use the analysis to study iteration limits for exponential type maps of the complex plane.  
We consider the $L^{\infty}$ to $L^2$ comparison for $\mathcal H(R)$.  A comparison is an ingredient in the work of Liu-Sun-Yau \cite[Lemma 4.3]{LSY1}.  The norm comparison is an ingredient in the work of Burns-Masur-Wilkinson \cite[Section 5.1]{BMW} on the geodesic flow.  The comparison is basic for the Teo approach for lower bounds for curvature \cite[Proposition 3.1]{LPT}; also see Theorem \ref{wppastbd} below and the remarks on the Teo bounds.  The comparison is used to study the covariant derivatives of the gradient of geodesic-length in \cite{Wlext}.  The Axelsson-Schumacher bound for the Hessian of geodesic-length is presented in terms of the $L^{\infty}$ norm \cite{AxSchu}; relating the bound to WP length involves the ratio of norms.  We give the comparison of norms in terms of the surface systole $\Lambda(R)$, the shortest length for a closed geodesic on $R$.  

\begin{corollary}\textup{Comparison of norms.}\label{pantsbd}
On a Bers region $\caB(\caP)$, the WP Hermitian pairing is uniformly comparable to $\sum_{\alpha\in\caP}|\partial\lla^{1/2}|^2$.  Equivalently on a Bers region, the norms $\|\sum_{\alpha\in\caP}a^{\alpha}\lambda\ssa\|^2_{WP}$ and $\sum_{\alpha\in\caP}|a^{\alpha}|^2$ are uniformly comparable.  Given $\epsilon>0$, there is a positive value $\Lambda_0$, such that for the surface systole $\Lambda(R)\le\Lambda_0$, the maximal ratio of $L^{\infty}$ and $L^2$ norms for $\caH(R)$ satisfies 
\[
(1-\epsilon)\Big(\frac{2}{\pi\Lambda(R)}\Big)^{1/2}\,\le\,\max_{\mu\in\caH(R)}\frac{\|\mu\|_{\infty}}{\|\mu\|_{WP}}\,\le\,(1+\epsilon)\Big(\frac{2}{\pi\Lambda(R)}\Big)^{1/2}.
\]
The maximal ratio is approximately realized for the gradient of the shortest geodesic-length.
\end{corollary}
\begin{proof}
The first comparison is established in \cite[Theorem 4.3]{Wlbhv}.  The argument combines the  linear independence of geodesic-lengths for a pants decomposition and the limiting of geodesic root-lengths to a $(2\pi)^{-1/2}$-multiple of an orthonormal frame for lengths small. 
It is also established in the proof that on a Bers region the matrix of main coefficients
\[
\mathbf A\,=\,\Big(a\ssa(\beta)\lla^{-1/2}\llb^{1/2}\Big)_{\alpha,\beta\in\caP}\,=\,\Big(4\langle\lambda\ssa,\lambda\ssb\rangle\Big)_{\alpha,\beta\in\caP}
\]
varies in a compact set in $GL(\mathbb R)$.   The pairing formula
\[
2\partial\llb^{1/2}(\sum_{\alpha\in\caP}a^{\alpha}\lambda\ssa)\,=\,\sum_{\alpha\in\caP}a^{\alpha}a\ssa(\beta)\lla^{-1/2}\llb^{1/2}
\]
provides that the Hermitian forms of the second stated comparison differ by conjugation by $\mathbf A$; the second pair of norms are comparable. 

To consider the $L^{\infty}$ norm, we choose a sufficiently small positive constant $c_0$ and consider the set $\sigma$ of disjoint geodesics $\alpha$ with $\lla\le c_0$.  A harmonic Beltrami differential is given as a unique linear combination
\[
\mu\,=\,\sum_{\alpha\in\sigma}a^{\alpha}\lambda\ssa\,+\,\mu_0,\quad\mbox{for }\mu_0\perp\grad\lla, \alpha\in\sigma.
\]
By Theorem \ref{gradpair}, for the collars $c(\alpha),\alpha\in\sigma$, the main coefficients of $\mu_0$ vanish and from Proposition \ref{CSest}, the maximum of $|\mu_0|$ on the collars is bounded in terms of the maximum on the boundary.  By Lemma \ref{cuspbd}, the maximum on a cusp region is bounded in terms of the maximum on the boundary.   The maximum of $|\mu_0|$ is bounded in terms of the maximum on the complement of the collars and cusp regions. The complement $R'$ of the collars and cusp regions has injectivity radius bounded below by $c_0/2$.  Harmonic Beltrami differentials satisfy a mean value estimate: given $\epsilon>0,\,|\nu(p)|\le c_{\epsilon}\int_{B(p;\epsilon)}|\nu|dA$ for the hyperbolic area element.  On the region $R'$, the maximum is bounded in terms of the $L^1$ norm for a ball of diameter $c_0$, and the $L^1$ norm is bounded in terms of the $L^2$ norm for the ball.  In brief, $\|\mu_0\|_{\infty}$ is bounded by $c_1\|\mu_0\|_{WP}$ for a suitable constant depending on the choice of $c_0$. 

By Proposition \ref{CSest} and Corollary \ref{gradexp}, on a $\beta\in\sigma$ collar, the harmonic Beltrami differential $\mu$ satisfies for coefficients $a=(a^{\alpha})_{\alpha\in\sigma}$  
\[
\mu\,=\,\sum_{\alpha\in\sigma}a^{\alpha}a\ssa(\beta)\sin^2\theta/2\lla^{1/2}\,+\,
O((e^{-\theta/\llb}\,+\,e^{(\theta-\pi)/\llb})\|a\|\|\mu\|_{WP}),
\]
and the WP norm for the surface satisfies
\[
\|\mu\|_{WP}^2\,=\,\|\sum_{\alpha\in\sigma}a^{\alpha}\lambda\ssa\|_{WP}^2+\|\mu_0\|_{WP}^2.
\]
Modulo an overall $O(e^{-1/c_0})$ approximation, from the expansion, we can consider for the surface systole sufficiently small, that the ratio of norms is maximized for $\mu_0=0$, and that the $|\mu|$ maximum occurs on a geodesic $\beta\in\sigma$.   From Theorem \ref{gradpair}, the matrix $(a\ssa(\beta)/2\lla^{1/2})$ is given as diagonal with entries $1/\pi\lla^{1/2}$ and a remainder of $O(c_0^{5/2})$, while the matrix $(\langle\lambda\ssa,\lambda\ssb\rangle)$ is given as the $1/2\pi$ multiple of the identity and a remainder of $O(c_0^3)$.  Modulo an $O(c_0^{5/2})$ approximation, we consider only the matrix leading terms.  The ratio of contributions of leading terms is maximized for $\mu=\lambda_{\alpha'}$, for $\ell_{\alpha'}$ the surface systole, the shortest geodesic-length.  The value is $(2/\pi\ell_{\alpha'})^{1/2}$.  
\end{proof} 

Gradients of bounded geodesic-length functions are approximately determined by point evaluation on the geodesics of a bounded pants decomposition as follows. In the direction of the imaginary axis, the elementary differential is given as $\omega=\overline{(idy/iy)^2}(ds^2)^{-1}=1$.   By Corollary \ref{gradexp}, for geodesics $\alpha,\beta$ simple, disjoint, for the standard collar representation, the gradient of $\lla$ is given as $\grad\lla=2/\pi+O(\lla^3)$ along $\alpha$ and as $\grad\lla=O(\lla^2\llb)$ along $\beta$.  For a bounded pants decomposition, the geodesic-length gradients approximately diagonalize point evaluation on the geodesics.

By Proposition \ref{CSest} and Theorem \ref{gradpair}, holomorphic quadratic differentials in $\grad\lla^{\perp}\subset Q(R)$ also have their $L^p$, $1\le p\le\infty$, norms on the collar $c(\alpha)$ bounded by the supremum norms on the collar boundary.  This property can be combined with convergence of hyperbolic metrics to bound holomorphic quadratic differentials and their pairings on degenerating families of Riemann surfaces.  In particular for $\lla$ tending to zero on a family, the spaces $\grad\lla^{\perp}$ limit to the holomorphic quadratic differentials with finite $L^p$ norm in a neighborhood of the resulting cusp pair.         

\section{Green's functions for the operator $\Delta=-2(D-2)^{-1}$}

The deformation equation for a hyperbolic metric involves the Laplace-Beltrami operator $D$ acting on $L^2(R)$.  In particular the linearization of the constant curvature $-1$ equation involves the operator $-(D-2)$.  Starting from a harmonic Beltrami differential, solving for the deformed hyperbolic metric involves the operator $\Delta=-2(D-2)^{-1}$.  
The WP curvature tensor is given in terms of harmonic Beltrami differentials and the operator $\Delta$.  Accordingly the operator plays an important role in the works 
 \cite{Jost2,Jost1,Schucurv,Siuwp,LPT,Trap,Wfharm,Wfhess,Wlchern,Wlhyp}.  Wolf organizes the deformation calculation in a different manner, in effect using the Jacobians of harmonic maps for evaluation of the operator \cite{Wfharm}.  Huang uses the approach of Wolf to evaluate the operator \cite{Zh,Zh2,Zh3}.  Schumacher organizes the deformation calculation in a different manner \cite{AxSchu,Schu}.  He considers on the total space of a deformation family, the Chern form of the fiberwise K\"{a}hler-Einstein (hyperbolic) metric.  He finds that the negative Chern form is positive definite and observes that the horizontal lift of a deformation field is the corresponding harmonic Beltrami differential; the description does not involve potential theory.   In his approach the operator $\Delta$ plays a secondary role.   Liu-Sun-Yau follow the Schumacher approach; their work involves only limited consideration of the potential theory of $-(D-2)$ \cite{LSY1,LSY2,LSY3}.  We consider the basic potential theory for the operator.    

The Laplace-Beltrami operator is essentially self-adjoint acting on $L^2(R)$.   The integration by parts formula 
\[
\int_RfDg\,dA\,=\,-\int_R\nabla f\nabla g\,dA
\]
provides that the spectrum of $D$ is non positive and that $\Delta$ is a bounded operator acting on $L^2(R)$ with unit norm.  The maximum principle for the equation $(D-2)f=g$ provides that $2\max_R|f|\le\max_R|g|$, for $g$ continuous, vanishing at any cusps.  By a general argument, $f$ also vanishes at any cusps.   At a maximum $p$ of $f$, then $Df(p)\le0$ and consequently $2f(p)\le -g(p)$; at a minimum $q$ of $f$ then $Df(q)\ge 0$ and $2f(q)\ge -g(q)$ (if $f(q)$ is negative the inequality for the absolute value follows).   A consequence is that $\Delta$ is a bounded operator on $C_0(R)$ with at most unit norm; the equation $\Delta1=1$ and an approximation provide that the norm is unity.  The inequalities also provide that $f$ is non negative if $g$ is non positive.

We present the standard properties of the operator \cite{GT, Wells}.   

\begin{theorem}\textup{Properties of $\Delta$.}\label{delprop}  The operator is self-adjoint, positive and bounded on $L^2(R)$.  The operator is bounded on $C_0(R)$ with unit norm.   The operator has a positive symmetric integral kernel Green's function $G$.
\end{theorem}

The Green's function is given by a uniformization group sum
\[
G(z,z_0)\,=\,\sum_{A\in\Gamma}-2Q_2(d(z,Az_0))
\]
for $Q_2$ an associated Legendre function, and $d(\ ,\ )$ hyperbolic distance on $\mathbb H$ \cite{Fay}.  The summand $-2Q_2$ is the Green's function for the operator $\Delta$ acting on functions small at infinity on $\mathbb H$.  The summand has a logarithmic singularity at the origin and satisfies $-Q_2\approx e^{-2d(\ ,\ )}$ at large distance.

\begin{proposition}\label{grest}\textup{The distant sum estimate.}  On the Riemann surface $R$ the Green's function is bounded as
\[
G(z,z_0)\,\le\,C\,\inj(z_0)^{-1}e^{-d_R(z,z_0)}
\]
where for $c_0$ positive, the constant $C$ is uniform in the surface  $R$ for $d_R(z,z_0)\ge c_0$.
\end{proposition}
\begin{proof}  The elementary Green's function satisfies a mean value estimate
\[ -Q_2(z,z_0)\le c_{\epsilon}\int_{B(z_0;\epsilon)}-Q_2(z,w)dA
\] 
for $d(z,z_0)>\epsilon$.  The  inequalities are formally the same as for estimating the series $\Theta\ssa$,
\[
\sum_{A\in\Gamma}-2Q_2\le \sum_{A\in\Gamma}c_{\epsilon}\int_{B(Az_0;\epsilon)}-Q_2dA \le c_{\epsilon}\inj(z_0)^{-1}\int_{\cup_{A\in\Gamma}B(Az_0;\epsilon)}-Q_2dA.
\]
For the basepoint $z\in\mathbb H$, let $\delta=d(z,\ )$ be distance from the basepoint.  The integrand satisfies $-Q_2\approx e^{-2\delta}$; the area element satisfies $dA\approx e^{\delta}d\theta d\delta$ for $\theta$ the angle about the basepoint.   The union of balls is necessarily contained in the ball complement $\mathbb H - B(z;d_R(z,z_0)-\epsilon)$.   The integral of $-Q_2dA$ over the ball complement is bounded by $Ce^{-d_R(z,z_0)}$, as desired.  
\end{proof}

Analyzing the WP curvature tensor involves understanding the contribution of $\Delta\mua\overline\mub$ on a collar $c(\eta)$.  Following the Collar Principle \cite[Chapter 8, Section 2]{Wlcbms}, the main contribution is expected from the rotationally invariant on a collar main term of $\mua\overline\mub$; given the exponential factor in the remainder term of Corollary \ref{gradexp}, the main term is approximately the product of main terms of the individual geodesic-length gradients. We consider the main term. 

Consider the geodesic $\eta$ having the imaginary axis as a component of its lift to $\mathbb H$ and the polar coordinate $z=re^{i\theta}$ on $\mathbb H$.  A rotationally invariant function on the collar, lifts to a function of the single variable $\theta$.   The rotationally invariant component of $(D - 2)$ is the one-dimensional operator $(D_{\theta}-2)=\sin^2\theta \frac{d^2}{d\theta^2}-2$.  The operator is essentially self-adjoint on $L^2(0,\pi)$ for the measure $\csc^2\theta d\theta$.  The function $u(\theta)=1-\theta\cot\theta$ is positive on $(0,\pi)$, vanishes to second order at zero, and satisfies $(D_{\theta}-2)u=0$.  The Green's function for $(D_{\theta}-2)$ is determined by the conditions: $(D_{\theta}-2)\mathbf G(\theta,\theta_0)=0,$ $\theta\ne\theta_0$;  $\mathbf G(\theta,\theta_0)$ vanishes at the interval endpoints;  $\mathbf G(\theta,\theta_0)$ is continuous with a unit jump discontinuity in $\frac{d}{d\theta}\mathbf G(\theta,\theta_0)$ at $\theta=\theta_0$.  The one-dimensional Green's function is given as
\begin{equation*}
\mathbf G(\theta,\theta_0)\,=\,\frac{-1}{\pi}
\begin{cases}
\,u(\theta)u(\pi-\theta_0), & \theta\le\theta_0 \\
\,u(\pi-\theta)u(\theta_0), &  \theta_0\le \theta.\\ 
\end{cases}
\end{equation*}
For rotationally invariant functions $g$ on a collar, the analysis of $(D_{\theta}-2)^{-1}g$ can be effected in terms of integrals of $\mathbf G$.  In Section \ref{curvexp} we use a different approach based on the general properties of Theorem \ref{delprop}, the estimate of Proposition \ref{thetprop}, and the simple equation
\begin{equation}\label{dsin4}
(D_{\theta}-2)\sin^2\theta\,=\,-4\sin^4\theta.
\end{equation}
In particular for the elementary Beltrami differential $\omega$, we formally have the equation
\[
 2\Delta\omega\overline\omega\,=\,2\Delta\sin^4\theta\,=\,\sin^2\theta.
\]

\section{The curvature tensor}\label{tensorsec}

Bochner discovered general symmetries of the Riemann curvature tensor with respect to the complex structure $J$ for a K\"{a}hler metric \cite{Boch}.  The symmetries are revealed by complexifying the tensor and decomposing by tangent type.   The curvature operator is the commutator of covariant differentiation
\[
R(U,V)W\,=D_UD_VW\,-\,D_VD_UW\,-D_{[U,V]}W.
\]
The curvature tensor $\langle R(U,V)W,X\rangle$ is defined on tangent spaces $T$. Tangent spaces are complexified by tensoring with the complex numbers $\mathbb C$.   The complexification is decomposed into tangents of {\em holomorphic} and {\em anti holomorphic type} $\mathbb C\otimes T=T^{1,0}\oplus T^{0,1}$, by considering the $\pm i$-eigenspaces of $J$.  Complex conjugation provides a natural complex anti linear isomorphism $\overline{T^{0,1}}=T^{1,0}$.  Real tangent directions are given as sums $Z\oplus\overline Z$ for $Z$ in $T^{1,0}$.   A general tensor
is extended to the complexification by complex linearity.   Bochner found for K\"{a}hler metrics that the complexified curvature tensor has a block form relative to the tangent type decomposition.  The only non zero curvature evaluations are for 
$T^{1,0}\times\overline{T^{1,0}}\times T^{1,0}\times\overline{T^{1,0}}$ and the conjugate space, with the latter evaluation simply the conjugate of the former evaluation.  In brief for a 
K\"{a}hler metric, the Riemann curvature tensor is fully determined by the evaluations $R_{\alpha\overline\beta\gamma\overline\delta}$.

We follow Bochner's exposition \cite{Boch}.  Riemannian geometry formulas are presented in terms of the complexification.   In particular, for the local holomorphic coordinate $(z_1,\dots,z_n)$, formal variables $t_i$, $1\le i\le 2n$, range over $\{z_1,\dots,z_n;\overline{z_1},\dots,\overline{z_n}\}$.  The Riemannian metric is given as
\[
ds^2\,=\,\sum_{i,j\in\{1,\dots,n;\overline 1,\dots,\overline n\}}g_{ij}dt_idt_j\,=\,2\sum_{\alpha,\beta\in\{1,\dots,n\}}g_{\alpha\overline\beta}dz\ssa d\overline{z\ssb},
\]
where $\sum_{\alpha,\beta}g_{\alpha\overline\beta}dz\ssa d\overline{z\ssb}$ is the Hermitian form for holomorphic type tangents.  By convention, Roman indices $h,i,j,k$ vary over $\{1,\dots,n;\overline 1,\dots,\overline n\}$, while Greek indices $\alpha,\beta,\gamma,\delta$ vary over $\{1,\dots,n\}$.  
A two-dimensional surface element is given in the form
\[
t_i\,=\,a^ix\,+\,b^iy
\]
or equivalently 
\[
z_{\alpha}\,=\,a^{\alpha}x\,+\,b^{\alpha}y,\ 
\overline{z_{\alpha}}\,=\,\overline{a^{\alpha}}x\,+\,\overline{b^{\alpha}}y,
\]
for $x,y$ real parameters and complex arrays $(a^1,\dots,a^n),\ (b^1,\dots,b^n)$, linearly independent over $\mathbb R$.  The Riemannian sectional curvature of the surface element is
\begin{equation}\label{seccurv}
\frac{\sum_{h,i,j,k}R_{h,i,j,k}a^hb^ia^jb^k}{\sum_{h,i,j,k}(g_{hj}g_{ik}\,-\,g_{hk}g_{ij})a^hb^ia^jb^k}
\end{equation}for the complexified tensor.  A holomorphic surface element is given in the simple form
\[
z\ssa\,=\,a^{\alpha}z
\]
for $z$ a complex parameter and $(a^1,\dots,a^n)$ a non zero complex array. 
The Riemannian sectional curvature of the holomorphic surface element is
\begin{equation}\label{holcurv}
\frac{-2\sum_{\alpha,\beta,\gamma,\delta}R_{\alpha\overline\beta\gamma\overline\delta}a^{\alpha}\overline{a^{\beta}}a^{\gamma}\overline{a^{\delta}}}
{\sum_{\alpha,\beta,\gamma,\delta}(g_{\alpha\overline\beta}g_{\gamma\overline\delta}\,+\,
g_{\gamma\overline{\beta}}g_{\alpha\overline\delta})a^{\alpha}\overline{a^{\beta}}a^{\gamma}\overline{a^{\delta}}}
\end{equation}
for the complexified tensor.   An example of the setup is provided by the upper half plane $\mathbb H$ with hyperbolic metric.   For the complex variable $z$, the Hermitian form for $T^{1,0}\mathbb H$ is $|dz|^2/2(\Im z)^2$ with Riemannian sectional curvature $-1$.

The holomorphic cotangent space at the point $R$ of Teichm\"{u}ller space is represented by the holomorphic quadratic differentials $Q(R)$.   In the investigation of WP distance and geodesic-length functions \cite{BMW,DW2,Msext,Rier,Wfhess,Wlcomp,Wlbhv,Wlext,Wlcbms} the WP Riemannian cometric is given by the real part of the Petersson product for $Q(R)$.   Following the Bochner setup, the Hermitian form corresponding to the WP Riemannian pairing is {\em one-half} the Petersson Hermitian form.  The original calculation of the curvature tensor is for the Petersson Hermitian form \cite[see Definition 2.6 and Theorem 4.2]{Wlchern}.  Consistent with the investigation of WP distance and geodesic-length functions, the WP curvature tensor corresponding to one-half the Petersson Hermitian form is
\begin{equation}\label{WPcurv}
R(\mu,\nu,\rho,\sigma)\,=\,\frac12\int_R\mu\overline\nu\Delta\rho\overline\sigma\,dA\,+\,\frac12\int_R\mu\overline\sigma\Delta\rho\overline\nu\,dA,
\end{equation}
for harmonic Beltrami differentials $\mu,\nu,\rho,\sigma\in\mathcal H(R)$ \cite{Wlchern}.

We first review some results about the sectional curvatures.  A collection of authors have studied the curvature \cite{Ahcurv,Zh4,Zh,Zh2,Zh3,Jost2,Jost1,LSY1,LSY2,LSY3,Royicm,Schu,Schucurv,Schucurv2,Siuwp, LPT,Trap,Trcurv,Trmbook,Wlchern,Wlhyp}. 
 Representative curvature bounds are presented below for the above normalization.  We write $\chi(R)$ for the Euler characteristic of the surface, and $\dim\caT$ for the complex dimension of the Teichm\"{u}ller space.  We write $\Lambda(R)$ for the systole of the surface, the shortest length for a closed geodesic.   The minimal injectivity radius for the complement of standard cusp regions is realized as either one half the surface systole or as unity on the boundary of a cusp region.  

\begin{theorem}\textup{WP curvature bounds.}\label{wppastbd}  The sectional curvature is negative, \cite{Trcurv,Wlchern}.  The holomorphic and Ricci curvatures are bounded above 
by $2(\pi\chi(R))^{-1}$ \cite{Royicm,Wlchern}. 
The scalar curvature is bounded above by $2(\dim\caT)^2(\pi\chi(R))^{-1}$, \cite{Wlchern}. 
Given $\Lambda_0>0,$ there is a positive constant $C$, independent of topological type, such that at $R\in\caT$ with $\Lambda(R)\ge\Lambda_0$, sectional curvatures are bounded below by $-C$ \cite{Zh3}.   
For $\dim\caT>1$, there are positive constants $C,\,C'$, depending on the topological type, such that at $R\in\caT$, sectional curvatures are as large as $-C\Lambda(R)$\footnote[1]{Huang has withdrawn his statement of a general upper bound for sectional curvatures \cite{Zh5}.} and sectional curvatures are bounded below by $-C'(\Lambda(R))^{-1}$, \cite{Zh2}.   
There is a universal function $\mathbf c(\Lambda)$ of the surface systole, such that the sectional and Ricci curvatures are bounded below by $-\mathbf c(\Lambda(R))^2$, and the scalar curvature is bounded below by $-\dim\caT\mathbf c(\Lambda(R))^2$ \cite[Proposition 3.4]{LPT}.  The universal function satisfies for $\Lambda$ small, 
$\mathbf c(\Lambda)\approx 8^{1/2}(\pi^{1/2}\Lambda)^{-1}$, \cite[Proposition 3.4 and formula (3.7)]{LPT}.       
\end{theorem}

To simplify and make explicit the WP metric and curvature tensor, we use the gradients of geodesic-length functions.  For a pants decomposition $\mathcal P$, a maximal collection of disjoint simple closed geodesics, the gradients $\{\grad\ell_{\alpha}\}_{\alpha\in\caP}$ provide a global $\mathbb C$-frame for $T^{1,0}\caT$.  The root-length gradients $\{\lambda_{\alpha}\}_{\alpha\in\caP}$, $\,\lambda\ssa=\grad\ell\ssa^{1/2}$, provide a global $\mathbb C$-frame that limits to an orthogonal frame for vanishing lengths.  To provide uniform remainder terms, we use bounded pants decomposition regions, Bers regions.  In particular the WP metric is described on $\caT$ by considering expansions for bounded pants decompositions.  The expansion for the Hermitian form on $T^{1,0}\caT$ (one-half the Petersson Hermitian form) is
\begin{equation}\label{wpherm}
\langle\langle\lambda\ssa,\lambda\ssb\rangle\rangle\,=\,\frac{\delta_{\alpha\beta}}{4\pi}\,+\,O(\lla^{3/2}\llb^{3/2})
\end{equation}
for $\alpha,\beta$ simple geodesics, coinciding or disjoint, and given $c_0>0$, the remainder term constant is uniform for $\lla,\llb\le c_0$.  We have the following for the curvature tensor. 

\begin{theorem}\label{mainthm}
The WP curvature tensor evaluation for the root-length gradient for a simple closed geodesic satisfies
\[
R(\lambda\ssa,\lambda\ssa,\lambda\ssa,\lambda\ssa)\,=\,\frac{3}{16\pi^3\lla}\,+\,O(\lla),
\]
and
\[
R(\lambda\ssa,\lambda\ssa,\ ,\ )\,=\,\frac{3|\langle\langle\lambda\ssa,\ \rangle\rangle|^2}{\pi\lla}\,+\,O(\lla\|\ \|_{WP}^2).
\]
For simple closed geodesics, disjoint or coinciding, with at most pairs coinciding, the curvature evaluation $R(\lambda\ssa,\lambda\ssb,\lambda\ssg,\lambda\ssd)$ is bounded as $O((\lla\llb\llg\lld)^{1/2})$. Furthermore for $\alpha,\beta$ disjoint with $\lla,\llb\le\epsilon$, the evaluations $R(\lambda\ssa,\lambda\ssa,\lambda\ssb,\lambda\ssb)$ and 
$R(\lambda\ssa,\lambda\ssb,\lambda\ssa,\lambda\ssb)$ are bounded as $O(\epsilon^4)$.   For $c_0$ positive, the remainder term constants are uniform in the surface $R$ and independent of the topological type for $\lla,\llb,\llg,\lld\le c_0$.
\end{theorem}
\noindent Theorem \ref{mainthm} is established in the next section in the form of Theorem \ref{Iexp}.
  
\begin{corollary}\label{wpseccurv}
The root-length holomorphic sectional curvature satisfies 
\[
K(\lambda\ssa)\,=\,\frac{-3}{\pi\lla}\,+\,O(\lla).
\]
Given $\epsilon >0$, there is a positive constant $c_{g,n,\epsilon}$  such that for the 
surface systole $\Lambda(R)$, the sectional curvatures at $R\in\caT$ are bounded below by 
\[
\frac{-3-\epsilon}{\pi\Lambda(R)}\,-\, c_{g,n,\epsilon}.
\]
For $\alpha,\beta$ disjoint with $\lla,\llb\le\epsilon$, the sections spanned by 
$(J)\lambda\ssa,(J)\lambda\ssb$ have curvature bounded as $O(\epsilon^4)$.  
For $\alpha,\beta,\gamma,\delta$ simple closed geodesics, disjoint or coinciding, not all the same, the evaluation $R(\lambda\ssa,\lambda\ssb,\lambda\ssg,\lambda\ssd)$ is bounded as $O((\lla\llb\llg\lld)^{1/6})$.  For $c_0$ positive, the remainder term constants are uniform in the surface $R$ and independent of the topological type for $\lla,\llb,\llg,\lld\le c_0$.  
\end{corollary} 
\begin{proof} The holomorphic curvature expansion follows for the single index value $\alpha$ evaluation from formula (\ref{holcurv}), expansion (\ref{wpherm}) and the first expansion of Theorem \ref{mainthm}.  The curvature bound for the span of $(J)\lambda\ssa,(J)\lambda\ssb$ follows immediately from (\ref{seccurv}), the approximate orthogonality of the tangents and the corresponding statement of the theorem.    We next consider the general $O(\ell^{1/6})$ bound.   Theorem \ref{Iexp} is presented for geodesic-length gradients; to obtain bounds for root-length gradients $\lambda_*=\grad\ell_*^{1/2}$, we divide by $2\ell_*^{1/2}$.  For exactly three geodesics coinciding, the leading term for $I(\lambda\ssa,\lambda\ssa,\lambda\ssa,\lambda\ssb)$ has magnitude $O(\lla^{1/2}\llb^{3/2})$ from Theorem \ref{gradpair}, while for at most pairs coinciding, $I(\lambda\ssa,\lambda\ssb,\lambda\ssg,\lambda\ssd)$ is bounded as $O((\lla\llb\llg\lld)^{1/2})$.  The $O(\ell^{1/6})$ bound follows.

To establish the general sectional curvature lower bound, we choose a bounded pants decomposition $\caP$, and consider a general two-dimensional section by writing $\mu(a)=\sum_{\alpha\in\caP}a^{\alpha}\lambda\ssa,\, \mu(b)=\sum_{\alpha\in\caP}b^{\alpha}\lambda\ssa$ for a basis.  We assume the basis is orthogonal 
$\Re\langle\langle\mu(a),\mu(b)\rangle\rangle=0$.  The sectional curvature is given by formula (\ref{seccurv}).  We consider a lower bound.  The denominator for sectional curvature is 
\[
4\langle\langle\mu(a),\mu(a)\rangle\rangle\langle\langle\mu(b),\mu(b)\rangle\rangle\,-\,2\Im \langle\langle\mu(a),\mu(b)\rangle\rangle^2,
\]
and the Cauchy-Schwarz inequality provides a lower bound of 
$2\|\mu(a)\|^2\|\mu(b)\|^2$. We noted in Corollary \ref{pantsbd} that the WP Hermitian norm $\|\mu(a)\|$ and Euclidean Hermitian norm $\|a\|$ are uniformly comparable.  We consider the numerator for sectional curvature.  The bounds for the non diagonal evaluations and remainder bound for the diagonal evaluations provide an expansion of the numerator
\[
2\Re \sum\limits_{\alpha\in\caP}\frac{3}{16\pi^3\lla}(a^{\alpha}\overline{b^{\alpha}})^2\,+\,O_{g,n}(\|a\|^2\|b\|^2),
\]
where the remainder term constant depends only on the Bers constant. The remainder provides a uniformly bounded contribution to the sectional curvature. We proceed and consider the explicit sum.  The lower bound for the sum is negative, realized for $b^{\alpha}=ia^{\alpha}$.  To bound the sectional curvature, we use the lower bound $2\|\mu(a)\|^2\|\mu(b)\|^2$ for the denominator and bound the two norms from below. Given $\delta>0$, use expansion (\ref{wpherm}) to bound the contributions of cross terms, to show that
\[
\big\|\sum_{\alpha\in\caP}a^{\alpha}\lambda\ssa\big\|^2\,=\,\big\|\sum_{\lla<\delta}a^{\alpha}\lambda\ssa\big\|^2\,+\,\big\|\sum_{\lla\ge\delta}a^{\alpha}\lambda\ssa\big\|^2\,+\,O(\delta\|a\|^2),
\]
where the remainder is bounded by $C'\delta\|a\|^2$ for a suitable positive constant.  
Given $\epsilon>0$, for $\delta>0$ sufficiently small, from expansion (\ref{wpherm}) the first term on the right is bounded below by
\[
\frac{(1-\epsilon)}{4\pi}\,\sum_{\lla<\delta}|a^{\alpha}|^2.
\]
From the comparability of Hermitian norms, the second sum on the right is bounded below by
\[
\frac{C}{4\pi}\,\sum_{\lla\ge\delta}|a^{\alpha}|^2
\]
for a suitable positive constant.  We combine observations to find that
\[
\big\|\sum_{\alpha\in\caP}a^{\alpha}\lambda\ssa\big\|^2\,\ge\,\frac{(1-C'\delta)(1-\epsilon)}{4\pi}\,||| a |||^2,
\]
for $|||a|||^2=\sum_{\lla<\delta}|a^{\alpha}|^2\,+\,C\sum_{\lla\ge\delta}|a^{\alpha}|^2$ and positive constants.  The contribution of the explicit sum to sectional curvature is now bounded below by
\begin{multline*}
\frac{\Re\sum_{\alpha\in\caP}\frac{3}{\pi\lla}(a^{\alpha}\overline{b^{\alpha}})^2}{(1-C'\delta)^2(1-\epsilon)^2 |||a|||^2|||b|||^2}
\ge \frac{-3}{(1-C'\delta)^2(1-\epsilon)^2 \pi}\, \min{\{\frac{1}{\Lambda(R)},\frac{1}{C^2\delta}\}},
\end{multline*}
the desired final bound.  
\end{proof}

To study the geometry of the moduli space $\caM(R)$, McMullen introduced a K\"{a}hler hyperbolic metric with K\"{a}hler form of the form $\omega_{WP}\,+\,c\sum_{\alpha\in\caP}\partial\overline\partial\mbox{Log}\lla$ \cite{McM}.  He found that the metric is comparable to the Teichm\"{u}ller metric.  To study canonical metrics on the moduli space, including the Teichm\"{u}ller and complete K\"{a}hler-Einstein metrics, Liu-Sun-Yau used the negative WP Ricci form as a comparison and reference metric \cite{LSY1,LSY2,LSY6,LSY3,LSY7}.  To understand curvature of canonical metrics, Liu-Sun-Yau also used a combination of the WP metric and its negative Ricci form. Expansions for geodesic-length functions and WP curvature enable explicit comparisons.  We showed in \cite{Wlbhv} that $\log\lla$ is strictly pluri subharmonic with
\[
\partial\overline\partial\log\lla\,=
\,\frac{\partial\lla}{\lla}\frac{\overline\partial\lla}{\lla}\,+\,O_+(\lla\|\ \|^2)\,=
\,\frac{4|\langle\langle\lambda\ssa,\ \rangle\rangle|^2}{\lla} \,+\,O_+(\lla\|\ \|^2),
\]
for positive remainders, where for $c_0$ positive, the remainder term constant is uniform in $R$ for $\lla\le c_0$.  Theorem \ref{mainthm} provides an immediate comparison with $R(\lambda\ssa,\lambda\ssa,\ ,\ )$.  Ricci curvature is given as a sum $\sum_jR(\mu_j,\mu_j,\ ,\ )$ over a unitary basis of $\caH(R)$.  From the proof of Corollary \ref{pantsbd}, for a bounded pants decomposition, the transformation of $\{\lambda\ssa\}_{\alpha\in\caP}$ to unitary bases is given by elements from a bounded set in $GL(\mathbb R)$. In particular the Ricci form is uniformly comparable to a bounded pants decomposition sum 
$\sum_{\alpha\in\caP}R(\lambda\ssa,\lambda\ssa,\ ,\ )$. 

The Teo lower bound for sectional and Ricci curvatures is $-2(\max_{\caH(R)}\|\mu\|_{\infty}/\|\mu\|_{WP})^2$ \cite{LPT}.  She provides a universal bound (independent of topological type) for the ratio of norms, with the ratio bounded as $2/\pi^{1/2}\Lambda(R)$ for small surface systole.  The Corollary \ref{pantsbd} bound $(2/\pi\Lambda(R))^{1/2}$ for the ratio depends on topological type in the determination of $\Lambda_0$, but is optimal in $\Lambda$-dependence.

\section{Expansion of the curvature tensor}\label{curvexp}

We consider the WP curvature tensor evaluated on geodesic-length gradients. The contribution to the curvature integral (\ref{WPcurv}) from the thick subset of the surface $R$ is bounded by applying Proposition \ref{thetprop} and the supremum bound for $\Delta$ from Theorem \ref{delprop}.  The contribution of collars is found by evaluating main terms in the meridian Fourier series for functions on a collar and then applying supremum bounds.  

We introduce notation to simplify the statements.   For a simple closed geodesic lifted to the imaginary axis, the polar angle $\theta$ of $\mathbb H$ has an intrinsic definition $\theta\ssa$ on a collar $c(\alpha)$, through the distance formula $d(\alpha,p)=\log(\csc\theta\ssa(p)+|\cot\theta\ssa(p)|)$. On a collar the products $\mua\overline\mub$ and $\Delta\mua\overline\mub$ have expansions with main terms respectively $\sin^4\theta\ssa$ and $\sin^2\theta\ssa$.  We now write $\sin\ssa\theta$ for the restriction of the sine of $\theta\ssa$ to the collar $c(\alpha)$; near the collar boundary $\sin\ssa\theta$ is bounded as $O(\lla)$.  We begin with the expansion for $\Delta\mua\overline\mub$.

\begin{theorem}\label{delmuab}
With the above notation, the operator $\Delta$ acting on pairs of simple geodesic-length gradients has the expansions on $R$
\begin{equation}\label{delmuab1}
2\Delta \mua\overline\mua\,=\,a\ssa(\alpha)^2\sin\ssa^2\theta\,+\,O(\lla^2)
\end{equation}
and for $\beta$ disjoint from $\alpha$,
\begin{equation}\label{delmuab2}
2\Delta \mua\overline\mub\,=\,a\ssa(\alpha)a\ssb(\alpha)\sin\ssa^2\theta\,+\,a\ssa(\beta)a\ssb(\beta)\sin\ssb^2\theta\,+\,O(\lla^2\llb^2).
\end{equation}
For $c_0$ positive, the remainder term constants are uniform in the surface  $R$ for $\lla,\llb\le c_0$.
\end{theorem}
\begin{proof}  The collar $c(\alpha)$ is represented with the standard description in $\mathbb H$.  The approach is based on analyzing 
the $0^{th}$ meridian Fourier coefficient in the collar.   To apply equation (\ref{dsin4}) for a collar, we choose a smooth approximate characteristic function $\chi$ of $\mathbb R_+$ with support of the derivative $\chi'$ contained in $(-\log 2,0)$.  The function $\chi(\log(\sin\theta/\lla))$ is an approximate characteristic function of the collar $c(\alpha)$ in the half annulus $\mathcal A\ssa=\{1\le |z|<e^{\lla}\}$ for $z=re^{i\theta}\in\mathbb H$.  Consider the derivative equation 
\begin{multline*}
(D-2)\chi\sin^2\theta \\
	=\,\chi(D_{\theta}-2)\sin^2\theta+\chi'\sin^2\theta\,(3\cos^2\theta-\sin^2\theta)+\chi''\sin^2\theta\cos^2\theta \\
	=-4\chi\sin^4\theta\,+\,O(\lla^2),
\end{multline*}       
using equation (\ref{dsin4}) and that the support of $\chi',\chi''$ are bands about the collar boundary, where $\lla/2\le\sin\theta\le\lla$.  

For the first expansion we subtract the function $\chi a\ssa(\alpha)^2\sin^2\theta$ from $2\Delta\mua\overline\mua$ and consider an equation for the difference on a fundamental domain in $\mathcal A\ssa$, containing the collar $c(\alpha)$.  We find from applying the above expansion
\begin{multline}\label{d2delmua}
(D-2)(2\Delta\mua\overline\mua\,-\,\chi a\ssa(\alpha)^2\sin^2\theta)\\
=\,-4\mua\overline\mua\,+\,4\chi a\ssa(\alpha)^2\sin^4\theta\,+\,O(\lla^2).
\end{multline}
By Proposition \ref{thetprop}, on $R-c(\alpha)$ the first term on the right is bounded as $O(\lla^4)$. The support of the second term is contained in the collar $c(\alpha)$.  We apply Corollary \ref{gradexp} to bound the sum of the first two terms on $c(\alpha)$.  The exponential-sine remainder term function  $e^{-\pi\theta/2\ell}\sin\theta$ is decreasing on the interval $(\ell,\pi-\ell)$ with initial value bounded by $\ell$.  The remainder term function is also symmetric with respect to $\theta\rightarrow\pi-\theta$.  The remainder term of Corollary \ref{gradexp} contribution to $-4\mua\overline\mua+4\chi a\ssa(\alpha)^2\sin^4\theta$ is $O(e^{-2\pi\theta/\lla}\sin^4\theta)$, which we have is bounded as $O(\lla^4)$.   In summary, the right hand side of (\ref{d2delmua}) is pointwise bounded as $O(\lla^2)$.  Expansion (\ref{delmuab1}) follows by applying the operator $\Delta$ to the right hand side and applying the supremum bound of Theorem \ref{delprop}.   

The counterpart of equation (\ref{d2delmua}) for $\mua\overline\mub$ is in straightforward notation
\begin{multline}\label{d2delmuab}
(D-2)\big(2\Delta\mua\overline\mub-\chi a\ssa(\alpha)a\ssb(\alpha)\sin\ssa^2\theta-\chi a\ssa(\beta)a\ssb(\beta)\sin\ssb^2\theta\big)\\
=-4\mua\overline\mub+4\chi a\ssa(\alpha)a\ssb(\alpha)\sin\ssa^4\theta+4\chi a\ssa(\beta)a\ssb(\beta)\sin\ssb^4\theta + O(\lla^2\llb^2),
\end{multline}
where from Theorem \ref{gradpair}: $a\ssa(\alpha), a\ssb(\beta)$ are bounded, $a\ssb(\alpha)$ is bounded as $O(\lla\llb^2)$ and $a\ssa(\beta)$ is bounded as $O(\lla^2\llb)$.   By Proposition \ref{thetprop}, on $R-c(\alpha)-c(\beta)$ the first term on the right is $O(\lla^2\llb^2)$.  The supports of the second and third terms on the right are contained in the respective collars.   We again apply Corollary \ref{gradexp} to bound the sum of the terms on the collars.  On $c(\alpha)$ the remainder term contribution is bounded by $O((\llb^2/\lla^2+a\ssb(\alpha))(e^{-2\pi\theta/\lla}+e^{2\pi(\theta-\pi)/\lla})\sin^4\theta)$ and as above the exponential-sine product is bounded by $\lla^4$, leading to the overall bound of $O(\lla^2\llb^2)$.  The contributions on $c(\beta)$ are similarly bounded.  The desired expansion follows.  
\end{proof}

The following is a pointwise form of the expansion for the pairing $\langle\grad\lla,\grad\llb\rangle$.  

\begin{corollary}\label{muabbd}
With the above notation, the products of geodesic-length gradients have the expansions on $R$
\[
\mua\overline\mua\,=\,a\ssa(\alpha)^2\sin^4\ssa\theta\,+\,O(\lla^2)
\]
and for $\beta$ disjoint from $\alpha$,
\[
\mua\overline\mub\,=\,a\ssa(\alpha)a\ssb(\alpha)\sin^4\ssa\theta\,+\,a\ssa(\beta)a\ssb(\beta)\sin^4\ssb\theta\,+\,O(\lla^2\llb^2).
\]
For $\beta$ disjoint from $\alpha$, the product $\mua\overline\mub$ is bounded as $O(\lla\llb)$.  For $c_0$ positive, the remainder term constants are uniform in the surface  $R$ for $\lla,\llb\le c_0$.
\end{corollary}
\begin{proof}  The expansions follow from Corollary \ref{gradexp} and the exponential-sine bound $e^{-\pi\theta/2\ell}\sin\theta\le\ell$ on $(\ell,\pi-\ell)$.  The general bound 
for $\mua\overline\mub$ follows from the Theorem \ref{gradpair} bounds for $a\ssa$ and $a\ssb$.
\end{proof}

A basic calculation is for the variation of the area element by a deformation map.   Ahlfors found that the first variation of the hyperbolic area element vanishes for harmonic Beltrami differentials \cite{Ahsome}.  The second variation of the hyperbolic area element is an ingredient in the calculation of curvature for the WP metric and for the hyperbolic metric on the vertical line bundle of the universal curve \cite{Wlchern, Wlhyp}.   The second variation formula for both quasi conformal and harmonic deformation maps is
\[
\ddot{dA}[\mu,\mu]\,=\,\frac{d^2(f^{\epsilon\mu})^*dA}{d\epsilon^2}\bigg|_{\epsilon=0}\,=\,2(-\mu\overline\mu\,+\,\Delta\mu\overline\mu)\,dA,
\]
for $\mu\in\mathcal H(R)$ \cite{Wfharm, Wlchern, Wlhyp}.  Theorem \ref{delmuab} and Corollary \ref{muabbd} combine to provide an expansion for the second variation for geodesic-length gradients
\[
\ddot{dA}[\mua,\mua]\,=\,a\ssa(\alpha)^2(\sin\ssa^2\theta\,-\,2\sin\ssa^4\theta\,+\,O(\lla^2))\,dA.
\]
We have the following for the curvature integral.

\begin{theorem}\label{Iexp}
For $\alpha,\beta,\gamma,\delta$ simple closed geodesics, the integral
\[
I_{\alpha\beta\gamma\delta}\,=\,\int_R\mua\overline\mub\Delta\mug\overline\mud\,dA
\]
has the symmetries $I_{\alpha\beta\gamma\delta}=I_{\gamma\delta\alpha\beta}$ and 
$\overline{I_{\alpha\beta\gamma\delta}}=I_{\beta\alpha\delta\gamma}$.   The integral for the geodesics coinciding satisfies
\[
I_{\alpha\alpha\alpha\alpha}\,=\,\frac{3}{\pi^3}\lla\,+\,O(\lla^3),
\]
for three geodesics coinciding satisfies
\[
I_{\alpha\alpha\alpha\beta}\,=\,\frac{3}{2\pi^2}\langle\grad\lla,\grad\llb\rangle\,+\,O(\lla^4\llb^2),
\]
and all, at most pairs of geodesics coinciding, evaluations $I_{\alpha\beta\gamma\delta}$ are bounded as $O(\lla\llb\llg\lld)$.  The integrals $I_{\alpha\alpha\beta\beta}$ and $I_{\alpha\beta\alpha\beta}$ for coinciding pairs with $\lla=\llb$ are bounded as $O(\lla^6)$.  For $c_0$ positive, the remainder term constants are uniform in the surface  $R$ for $\lla,\llb,\llg,\lld\le c_0$.   
\end{theorem}
\begin{proof}  The first symmetry is a consequence of $\Delta$ being self-adjoint.  The second symmetry is immediate.  We develop a general expansion for the integrand and then consider cases based on the patterns of geodesics.   From Corollary \ref{gradexp} on a collar $c(\eta)$, we have
\[
\mua\overline\mub\,=\,a\ssa(\eta)a\ssb(\eta)\sin_{\sset}^4\theta\,+\,
O\big(\mathbf c_{\alpha\beta\eta}(e^{-2\pi\theta/\llet}+e^{2\pi(\pi-\theta)/\llet})\sin\sset^4\theta\big),
\]
for
\[
\mathbf c_{\alpha\beta\eta}\,=\,a\ssa(\eta)\llb^2/\llet^2+a\ssb(\eta)\lla^2/\llet^2+\lla^2\llb^2/\llet^4.
\]
We consider the collar integrals of the product with the terms in the $\Delta\mug\overline\mud$ expansion of Theorem \ref{delmuab}.    
The leading term of $\Delta\mug\overline\mud$ is $\sin^2\theta$.  The resulting product integral is
\begin{multline*}
\int_{c(\eta)}\mua\overline\mub\sin^2\theta dA\\ =\,a\ssa(\eta)a\ssb(\eta)\int_{c(\eta)}\sin^6\theta dA\,+\, O\big(\mathbf c_{\alpha\beta\eta}\int_{c(\eta)}(e^{-2\pi\theta/\llet}+e^{2\pi(\pi-\theta)/\llet})\sin^6\theta dA\big).
\end{multline*}
The collar is $c(\eta)=\{1\le r\le e^{\llet},\, \llet\le\theta\le\pi-\llet\}\subset\mathbb H$ and the hyperbolic area element is $dA=\csc^2\theta\,d\theta dr/r$.  For the first integral on the right the $\theta$  intervals $(0,\llet)$ and $(\pi-\llet,\pi)$ are included to find that
\[
\int_{c(\eta)}\sin^6\theta dA\,=\,\frac{3\pi}{8}\llet\,+\,O(\llet^6).
\]
The second integral on the right is symmetric with respect to $\theta\rightarrow\pi-\theta$.  The inequality $\sin\theta\le \theta$ is applied to find that
\[
\int_{c(\eta)}e^{-2\pi\theta/\llet}\sin^6\theta dA\,=\,O(\llet^6).
\]
We combine these considerations to find that
\begin{equation}\label{muabsin}
\int_{c(\eta)}\mua\overline\mub \sin^2\theta dA\,=\,\frac{3\pi}{8}a\ssa(\eta)a\ssb(\eta)\llet\,+\,O_{\alpha\beta\eta},
\end{equation}
for $O_{\alpha\beta\eta}=a\ssa(\eta)a\ssb(\eta)\llet^6+\mathbf c_{\alpha\beta\eta}\llet^6$.   
By Theorem \ref{gradpair}, the remainder $O_{\alpha\beta\eta}$ is: $O(\lla^4\llet^2)$ for $\alpha=\beta$, $\eta$ coinciding with $\alpha$ or disjoint, and is $O(\lla^2\llb^2\llet^2)$ for $\alpha$ disjoint from $\beta$, $\eta$ coinciding with one of $\alpha,\beta$ or not.  In all cases 
$O_{\alpha\beta\eta}$ is bounded as $O(\lla^2\llb^2\llet^2)$.  By the same approach the collar integral corresponding to the product with the remainder term of $\Delta\mug\overline\mud$ is
\begin{multline}\label{muab1}
\int_{c(\eta)}\mua\overline\mub dA 
\\=\,a\ssa(\eta)a\ssb(\eta)\int_{c(\eta)}\sin^4\theta dA+
O\big(\mathbf c_{\alpha\beta\eta}\int_{c(\eta)}(e^{-2\pi\theta/\llet}+e^{2\pi(\pi-\theta)/\llet})\sin^4\theta dA\big)
\\ =\,O\big(a\ssa(\eta) a\ssb(\eta)(\llet+\llet^4)+\mathbf c_{\alpha\beta\eta}\llet^4\big)\,=\,O(a\ssa(\eta) a\ssb(\eta)\llet)\,+\,O_{\alpha\beta\eta}/\llet^2,
\end{multline}  
for the same remainder bound $O_{\alpha\beta\eta}$.

We are ready to combine the expansions and estimates to find the contributions for the possible patterns of geodesics.  In general for the evaluation of $I_{\alpha\beta\gamma\delta}$, the expansions of Theorem \ref{delmuab} for $\Delta\mu_*\overline{\mu_*}$ are applied for the indices $\gamma\delta$ and the expansions (\ref{muabsin}) and (\ref{muab1}) 
for $\mu_*\overline{\mu_*}$ are applied for the indices $\alpha\beta$.

\underline{The main cases $\alpha\alpha\alpha\alpha$ and $\alpha\beta\alpha\alpha$.}   Expansion (\ref{delmuab1}) is combined with expansions (\ref{muabsin}) and (\ref{muab1}) to find the contribution of the collar $c(\alpha)$.  The leading term is
\[
\frac12a\ssa(\alpha)^2\int_{c(\alpha)}\mua\overline\mub\sin^2\theta dA\,=\,\frac{3\pi}{16}a\ssa(\alpha)^3a\ssb(\alpha)\lla\,+\,O_{\alpha\beta\alpha}.
\]
The remainder is bounded as $O(\lla^4\llb^2)$.   Using expansion (\ref{muab1}), the contribution of the remainder term from expansion (\ref{delmuab1}) is bounded as $O(\lla^3)$ for $\alpha=\beta$ and as $O(\lla^4\llb^2)$ for $\alpha,\beta$ disjoint.    
The contribution of $R-c(\alpha)$ is bounded by combining expansion (\ref{delmuab1}) with Proposition \ref{thetprop}.  The contribution is bounded as $O(\lla^4\llb^2)$.   Finally Theorem \ref{gradpair} is applied to evaluate the product $a\ssa(\alpha)^3a\ssb(\alpha)$ of coefficients for $\alpha$ and $\beta$ either coinciding or disjoint.

\underline{The cases $\beta^2\alpha^2$ and $\beta\gamma\alpha^2$ for $\alpha$ distinct from 
$\beta,\gamma$.}   Expansion (\ref{delmuab1}) is combined with expansions (\ref{muabsin}) 
and (\ref{muab1}) to find the collar contribution.  The leading term provides an integral 
\[
\frac12a\ssa(\alpha)^2\int_{c(\alpha)}\mub\overline\mug\sin^2\theta dA\,=\,\frac{3\pi}{16}a\ssa(\alpha)^2a\ssb(\alpha)a\ssg(\alpha)\lla\,+\,O_{\beta\gamma\alpha},
\]
which is bounded as $O(\lla^2\llb^2\llg^2)$.  The remainder bound from (\ref{delmuab1}) combines with expansion (\ref{muab1}) to provide the remaining collar bound.  The contribution of $R-c(\alpha)$ is bounded by combining expansion (\ref{delmuab1}) and Proposition \ref{thetprop}.

\underline{The remaining cases $\alpha\beta\alpha\beta$, $\alpha\beta\alpha\gamma$ and $\alpha\beta\gamma\delta$.}   The Corollary \ref{muabbd} bound that $\mua\overline\mub$, $\alpha$ disjoint from $\beta$, is bounded as $O(\lla\llb)$ and boundedness of the operator $\Delta$ in 
$C_0$ provide the desired bounds.  
\end{proof}

We remark that in Theorem \ref{delmuab} the remainder terms can be improved to include a factor of the inverse exponential-distance to the thick subset of $R$.  Specifically the $O(\lla^2)$ term on the right side of (\ref{d2delmua}) and the $O(\lla^2\llb^2)$ term on the right side of (\ref{d2delmuab}) have compact support in a neighborhood of the collar boundary, where the injectivity radius is bounded below by a positive constant.   To estimate the operator $\Delta$ applied to the remainders, Proposition \ref{grest} is applied with $z_0$ chosen in the remainder support.  The resulting estimate includes an exponential-distance factor. The explicit terms on the right hand sides of (\ref{d2delmua}) and (\ref{d2delmuab}) are bounded by the exponential-sine function $e^{-2\pi\theta/\lla}\sin^4\theta$.  To estimate the operator $\Delta$ applied to the explicit terms, Proposition \ref{grest} is applied with $z_0$ corresponding to the variable $\theta$, and the $\theta$ integral is estimated directly.  The resulting estimate includes an exponential-distance factor.  The improved form of (\ref{delmuab1}) leads to the improved remainder $O(\lla^4)$ for the $I_{\alpha\alpha\alpha\alpha}$ expansion.     

\section{Continuity of the pairing and curvature tensor}
\label{conWP}

We follow the general exposition of \cite{Wlext,Wlcbms}.  The WP completion of Teichm\"{u}ller space is the augmented Teichm\"{u}ller space $\Tbar$.  The partial compactification $\Tbar$ is described in terms of Fenchel-Nielsen coordinates and in terms of the Chabauty topology for $PSL(2;\mathbb R)$ representations.  The strata of $\Tbar$ correspond to collections of {\em vanishing lengths} and describe Riemann surfaces with nodes in the sense of Bers, \cite{Bersdeg}.  We use root-length gradients as a local frame for the tangent bundle for a neighborhood of a stratum point.   The behavior of the pairing for a frame is presented in (\ref{wpherm}) and the behavior of the connection for a frame is presented in \cite[Theorem 4.6]{Wlext}.  

The points of the Teichm\"{u}ller space $\mathcal T$ are equivalence classes 
$\{(R,f)\}$ of surfaces with reference homeomorphisms $f:F\rightarrow R$.  
The {\em complex of curves} $C(F)$ is defined as follows.  The vertices of
$C(F)$ are the free homotopy classes of homotopically nontrivial, non peripheral,
simple closed curves on $F$.  An edge of the complex consists of a pair of
homotopy classes of disjoint simple closed curves.  A $k$-simplex consists of
$k+1$ homotopy classes of mutually disjoint simple closed curves.  A maximal
simplex is a pants decomposition. The mapping class group $Mod$ acts on the complex $C(F)$.  

The Fenchel-Nielsen coordinates for $\mathcal T$ are given in terms of geodesic-lengths and
lengths of auxiliary geodesic segments, \cite{Abbook,Busbook,ImTan}.  A pants decomposition
 $\mathcal{P}=\{\alpha_1,\dots,\alpha_{3g-3+n}\}$ decomposes the
reference surface $F$ into $2g-2+n$ components, each homeomorphic to
a sphere with a combination of three discs or points removed.  A marked
Riemann surface $(R,f)$ is likewise decomposed into pants by the geodesics
representing the elements of $\mathcal P$.  Each component pants, relative to its hyperbolic
metric, has a combination of three geodesic boundaries and cusps.  For each
component pants the shortest geodesic segments connecting boundaries determine
 designated points  on each boundary.  For each geodesic $\alpha$ in the pants
decomposition of $R$, a parameter $\tau_{\alpha}$ is defined as the displacement along
the geodesic between designated points, one  for each side of the geodesic.
For marked Riemann surfaces close to an initial reference marked Riemann surface, the
displacement $\tau_{\alpha}$ is the distance between the designated points; in
general the displacement is the analytic continuation (the lifting) of the
distance measurement.  For $\alpha$ in $\mathcal P$ define the {\em
Fenchel-Nielsen  angle} by $\vartheta_{\alpha}=2\pi\tau_{\alpha}/\ell_{\alpha}$.
The Fenchel-Nielsen coordinates for Teichm\"{u}ller space for the
decomposition $\mathcal P$ are
$(\ell_{\alpha_1},\vartheta_{\alpha_1},\dots,\ell_{\alpha_{3g-3+n}},\vartheta_{\alpha_{3g-3+n}})$.
The coordinates provide a real analytic equivalence of $\mathcal T$ to
$(\mathbb{R}_+\times \mathbb{R})^{3g-3+n}$, \cite{Abbook,Busbook,ImTan}.  Each pants decomposition gives rise to a Fenchel-Nielsen coordinate system.  

A partial compactification is introduced by extending the range
of the Fenchel-Nielsen parameters.  The added points correspond to unions of
hyperbolic surfaces with formal pairings of cusps.  The interpretation of {\em length
vanishing} is the key ingredient.  For an $\ell_{\alpha}$ equal to zero, the
angle $\vartheta_{\alpha}$ is not defined and in place of the geodesic for
$\alpha$ there appears a pair of cusps; the reference map $f$ is now a homeomorphism of
$F-\alpha$ to a union  of hyperbolic surfaces (curves parallel to $\alpha$
map to loops encircling the cusps).  The parameter space for a pair
$(\ell_{\alpha},\vartheta_{\alpha})$ will be the identification space $\mathbb{R}_{\ge
0}\times\mathbb{R}/\{(0,y)\sim(0,y')\}$.  More generally for the pants decomposition
$\mathcal P$, a frontier set $\mathcal{F}_{\mathcal P}$ is added to the
Teichm\"{u}ller space by extending the Fenchel-Nielsen parameter ranges: for
each $\alpha\in\mathcal{P}$, extend the range of $\ell_{\alpha}$ to include
the value $0$, with $\vartheta_{\alpha}$ not defined for $\ell_{\alpha}=0$.  The
points of $\mathcal{F}_{\mathcal P}$ parameterize unions of Riemann
surfaces with each $\ell_{\alpha}=0,\alpha\in\mathcal{P},$ specifying a pair
of cusps. The points of $\mathcal F_{\mathcal P}$ are Riemann surfaces with nodes in the sense of Bers.   For a simplex $\sigma\subset\mathcal{P}$, define the
$\sigma$-null stratum, a subset of $ \mathcal F_{\mathcal P}$, as $\mathcal{T}(\sigma)=\{R\mid \ell_{\alpha}(R)=0\mbox{ iff }\alpha\in\sigma\}$. Null strata are given as products of lower dimensional Teichm\"{u}ller spaces.  The frontier set $\mathcal{F}_{\mathcal P}$
is the union of the $\sigma$-null strata for the sub simplices of
$\mathcal{P}$.  Neighborhood bases for points of $\mathcal{F}_{\mathcal P}
\subset\mathcal{T}\cup\mathcal{F}_{\mathcal P}$  are specified by the
condition that for each simplex $\sigma\subset\mathcal P$  the projection
$((\ell_{\beta},\vartheta_{\beta}),\ell_{\alpha}):
\mathcal{T}\cup\mathcal{T}(\sigma)\rightarrow\prod_{\beta\notin\sigma}(\mathbb{R}_+\times\mathbb{R})\times\prod_{\alpha\in\sigma}(\mathbb{R}_{\ge
0})$  is continuous.    For a simplex $\sigma$ contained in pants  decompositions $\mathcal P$ and $\mathcal P'$ the specified neighborhood systems for $\mathcal T \cup\mathcal{T}(\sigma)$ are equivalent.  The {\em augmented Teichm\"{u}ller space} $\Tbar=\mathcal{T}\cup_{\sigma\in
C(F)}\mathcal{T}(\sigma)$ is the resulting stratified topological space,
\cite{Abdegn, Bersdeg}.  $\Tbar$ is not locally compact since points of the
frontier do not have relatively compact neighborhoods; the  neighborhood bases are
unrestricted in the $\vartheta_{\alpha}$ parameters for $\alpha$ a $\sigma$-null.
The action of $Mod$ on $\mathcal T$ extends to an action by homeomorphisms on
$\Tbar$ (the action on $\Tbar$ is not properly discontinuous) and the quotient
$\Tbar/Mod$ is topologically the compactified moduli space of stable curves, \cite[see Math. Rev.56 \#679]{Abdegn}. 

There is an alternate description of the frontier points in terms of representations of groups and the Chabauty topology.  A Riemann surface with punctures and hyperbolic metric is uniformized by a cofinite subgroup $\Gamma\subset PSL(2;\mathbb R)$.  A puncture corresponds to the $\Gamma$-conjugacy class of a maximal parabolic subgroup.  In general a Riemann surface with labeled punctures corresponds to the $PSL(2;\mathbb R)$ conjugacy class of a tuple $(\Gamma,\langle\Gamma_{01}\rangle ,\dots,\langle\Gamma_{0n}\rangle )$ where $\langle\Gamma_{0j}\rangle $ are the maximal parabolic classes and a labeling for punctures is a labeling for conjugacy classes.  A {\em Riemann surface with nodes} $R'$ is a finite collection of $PSL(2;\mathbb R)$ conjugacy classes of tuples $(\Gamma^\ast,\langle\Gamma_{01}^\ast\rangle ,\dots,\langle\Gamma_{0n^\ast}^\ast\rangle )$ with a formal pairing of certain maximal parabolic classes.  The conjugacy class of a tuple is called a {\em part} of $R'$.  The unpaired maximal parabolic classes are the punctures of $R'$ and the genus of $R'$ is defined by the relation $Total\ area=2\pi(2g-2+n)$.  A cofinite $PSL(2;\mathbb R)$ injective representation of the fundamental group of a surface is topologically allowable provided peripheral elements correspond to peripheral elements.  A point of the Teichm\"{u}ller space $\mathcal T$ is given by the $PSL(2;\mathbb R)$ conjugacy class of a topologically allowable injective cofinite representation of the fundamental group $\pi_1(F)\rightarrow\Gamma\subset PSL(2;\mathbb R)$.  For a simplex $\sigma$, a point of $\mathcal T(\sigma)$ is given by a collection $\{(\Gamma^\ast,\langle\Gamma_{01}^\ast\rangle ,\dots,\langle\Gamma_{0n^\ast}^\ast\rangle )\}$ of tuples with: a bijection between $\sigma$ and the paired maximal parabolic classes; a bijection between the components $\{F_j\}$ of $F-\sigma$ and the conjugacy classes of parts $(\Gamma^j,\langle\Gamma_{01}^j\rangle ,\dots,\langle\Gamma_{0n^j}^j\rangle )$ and the $PSL(2;\mathbb R)$ conjugacy classes of topologically allowable isomorphisms $\pi_1(F_j)\rightarrow\Gamma^j$, \cite{Abdegn, Bersdeg}. We are interested in geodesic-lengths for a sequence of points of $\mathcal T$ converging to a point of $\mathcal T(\sigma)$.  The convergence of hyperbolic metrics provides that for closed curves of $F$ disjoint from $\sigma$, geodesic-lengths converge, while closed curves with essential $\sigma$ intersections have geodesic-lengths tending to infinity, \cite{Bersdeg, Wlhyp, Wlcbms}.  

We refer to the Chabauty topology to describe the convergence for the $PSL(2;\mathbb R)$ representations.  Chabauty introduced a topology for the space of discrete subgroups of a locally compact group, \cite{Chb}.   A neighborhood of $\Gamma\subset PSL(2;\mathbb R)$ is specified by a neighborhood $U$ of the identity in $PSL(2;\mathbb R)$ and a compact subset $K\subset PSL(2;\mathbb R)$.  A discrete group $\Gamma'$ is in the neighborhood $\mathcal N(\Gamma,U,K)$ provided $\Gamma'\cap K\subseteq\Gamma U$ and $\Gamma\cap K\subseteq\Gamma'U$.  The sets $\mathcal N(\Gamma,U,K)$ provide a neighborhood basis for the topology. We consider a sequence of points of $\mathcal T$ converging to a point of $\mathcal T(\sigma)$ corresponding to 
$\{(\Gamma^\ast,\langle\Gamma_{01}^\ast\rangle ,\dots,\langle\Gamma_{0n^\ast}^\ast\rangle )\}$.  Important for the present considerations is the following property.  Given a sequence of points of $\mathcal T$ converging to a point of $\mathcal T(\sigma)$ and a component $F_j$ of $F-\sigma$, there exist $PSL(2;\mathbb R)$ conjugations such that restricted to $\pi_1(F_j)$ the corresponding representations converge element wise to $\pi_1(F_j)\rightarrow\Gamma^j$,
 \cite[Theorem 2]{HrCh}.   

We consider the geometry of geodesic-length gradients in a neighborhood of an augmentation point of $\mathcal T(\sigma)\subset\overline{\mathcal T}$ \cite{Wlcbms}.  For the following, we refer to the elements of $\sigma$ as the {\em short geodesics}. In a neighborhood of a point of $\caT(\sigma)$, the WP metric is approximately a product of common metrics for the short geodesic-lengths and lower dimensional WP metrics; in particular
\[
\langle\ ,\ \rangle_{WP}\,=\,2\pi\sum_{\alpha\in\sigma}(d\lla^{1/2})^2+(d\lla^{1/2}\circ J)^2\,+\,\langle\ ,\ \rangle_{WP;\caT(\sigma)}\,+\,O(d_{\caT(\sigma)}^2\langle\ ,\ \rangle_{WP})
\]
for $J$ the complex structure of $\caT$ and $\langle\ ,\ \rangle_{WP;\caT(\sigma)}$ the lower dimensional metric of the stratum.   The strata of $\Tbar$ are geodesically convex and the distance to a stratum $\caT(\sigma)$ is given as
\[
d_{\overline{\caT(\sigma)}}\,=\,(2\pi\sum_{\alpha\in\sigma}\lla)^{1/2}\,+\,O(\sum_{\alpha\in\sigma}\lla^{5/2}).
\]
In \cite{Wlbhv} we studied the behavior of the covariant derivatives of the geodesic root-length gradients $D_U\lambda\ssa$ and showed that at first-order the metric continues to behave as an approximate product \cite[Theorem 4.6]{Wlext}.  In particular for short geodesics $\alpha\ne\alpha'$, and $\beta$ a disjoint geodesic, $D_{\lambda\ssa}\lambda_{\alpha'}$ and  $D_{\lambda\ssb}\lambda\ssa$ vanish on $\caT(\sigma)$.  Also for $\beta,\beta'$ disjoint, non short geodesics, $D_{\lambda\ssb}\lambda_{\beta'}$ converges to its value on the limiting surface.  The formulas involving the complex structure $J$ follow by observing that the complex structure is parallel with respect to $D$, since the metric is K\"{a}hler.  

We now examine the metric second-order behavior by considering the curvature tensor.   We recall the convention that on $\caT(\sigma)$ the pairings $\langle\langle\lambda\ssa,\lambda\ssb\rangle\rangle$ vanish for $\alpha$ in $\sigma$ and $\beta$ disjoint and the pairings $\langle\langle\lambda\ssb,\lambda_{\beta'}\rangle\rangle$ vanish for geodesics on distinct components (parts) of a limiting noded Riemann surface.  We first recall a basic result \cite{Wlbhv}.  

\begin{lemma}\label{inprcts}
For a pants decomposition $\caP$ with subset of short geodesics $\sigma$, the pairing of geodesic root-length gradients is continuous in a neighborhood of a point of $\caT(\sigma)\subset\Tbar$.  The matrix $P$ of pairings for $\{\lambda\ssg\}_{\gamma\in\caP}$ determines a germ of a Lipschitz map from $\Tbar$ into a complex linear group $GL(\mathbb C)$.
\end{lemma}

We now consider the limiting values for the curvature tensor evaluated on the gradients for the pants decomposition $\caP$.  We introduce the convention that on $\caT(\sigma)$ the following evaluations $R(\lambda\ssa,\lambda\ssb,\lambda\ssd,\lambda\ssd)$ vanish: for $\alpha\in\sigma$ and at least one of $\beta,\gamma$ and $\delta$ distinct from $\alpha$; for $\alpha,\beta,\gamma,\delta\in\caP-\sigma$ and not all geodesics on the same component of a Riemann surface with nodes represented in $\caT(\sigma)$.  

\begin{theorem}\label{curvcont}
For a pants decomposition $\caP$ with subset of short geodesics $\sigma$, the diagonal curvature evaluations for $\alpha\in\sigma$ satisfy $R(\lambda\ssa,\lambda\ssa,\lambda\ssa,\lambda\ssa)=3(16\pi^3\lla)^{-1}\,+\,O(\lla)$ and all remaining curvature evaluations are continuous in a neighborhood of $\caT(\sigma)\subset\Tbar$.
\end{theorem}

\begin{proof}
For a pants decomposition the curvature evaluations involving short geodesics are treated in Theorem \ref{Iexp}.  It only remains to consider continuity for evaluations for geodesics in $\caP-\sigma$.  We choose a neighborhood $U$ in $\Tbar$ of $p\in\caT(\sigma)$ on which the geodesic-lengths $\llb$, $\beta\in\caP-\sigma$, are bounded away from zero.  The Chabauty topology provides that the geodesic-lengths $\llb$ are continuous on $U$.  We consider the curvature tensor evaluated on the geodesic-length gradients.   The evaluation is given by the integral pairing (\ref{WPcurv}).  Proposition \ref{thetprop} and Theorem \ref{delprop} combine to provide that for $\beta,\gamma,\delta\in\caP-\sigma$, the geodesic-length gradients $\mu\ssb$ and products $\Delta\mu\ssg\overline{\mu\ssd}$ are uniformly bounded on the neighborhood $U$.  Continuity of the integrals (\ref{WPcurv}) will follow from pointwise convergence of the integrands.  For $p$ corresponding to a noded Riemann surface with a part $R^{\natural}$ containing geodesics $\beta,\gamma$ and $\delta$, we establish pointwise convergence of $\mu\ssb$ and $\Delta\mu\ssg\overline{\mu\ssd}$ on $R^{\natural}$ and convergence to zero for any  component of the complement.

We establish pointwise convergence on $R^{\natural}$.   Points in $\caT$ near $p$ describe Riemann surfaces with $PSL(2;\mathbb R)$ representations of designated thick components close, modulo $PSL(2;\mathbb R)$ conjugation, to a uniformization of $R^{\natural}$.  Conjugate the representations to be close to the uniformization.   We will decompose the sums for $\mu\ssb$ and $\Delta\mu\ssg\overline{\mu\ssd}$ into terms for distance at most $\delta_0$ from a compact set in $\mathbb H$ and terms for greater distance.  The sum of latter terms will be uniformly bounded by applying the distant sum estimate as follows.   For a Chabauty neighborhood, the designated thick components are covered by a fixed compact set in $\mathbb H$.  The geodesics $\beta,\gamma$ and $\delta$ have lifts to axes intersecting the fixed compact set.   By the Quantitative Collar and Cusp Lemma and the method of Proposition \ref{thetprop} for the series $\overline{\Theta_*}(ds^2)^{-1}$, the absolute sum of terms at distance at least $\delta_0$ from the compact set is bounded as $O(e^{-\delta_0})$.  For Chabauty convergence, the bounded number of remaining terms of the series converge as the representations converge.  In summary the series $\Theta_*$ converge pointwise.   The operator $\Delta$ is given by the sum $\sum_{A\in\Gamma}-2Q_2(d(z,Az_0))$.   The individual terms of the sum converge as the representations converge and the sum is overall bounded in $L^1$.  It follows that the products $\Delta\mu\ssg\overline{\mu\ssd}$ converge pointwise.

In the case that the limiting Riemann surface with nodes has components distinct from $R^{\natural}$, then the union of short geodesics $\cup_{\alpha\in\sigma}\alpha$ separates the limiting surfaces into multiple components.  Principal components converge to $R^{\natural}$.  A point in a non principal component is at distance at least the half width $\log 2/\lla$ of the collars $c(\alpha),\,\alpha\in\sigma$, from the principal component.   For a non principal component conjugate the representations to converge.  
By Proposition \ref{thetprop} and the observation about distance, on a non principal component the gradients $\mu\ssb, \mu\ssg$ and $\mu\ssd$ have magnitude $O(\max_{\alpha\in\sigma}\lla)$.  Similarly by Proposition \ref{grest}, for points $z$ in the principal component and $z_0$ in a non principal component, $G(z,z_0)$ also has magnitude $O(\max_{\alpha\in\sigma}\lla)$.  The estimates for $\mu\ssg,\mu\ssd$, combine with the overall $L^1$ bound for the Green's function and the given bound for the Green's function to provide that the products  $\Delta\mu\ssg\overline{\mu\ssd}$ converge to zero on non principal components.   

\end{proof}

The formulas for the metric, covariant derivative and curvature tensor display an asymptotic product structure for an extension of the tangent bundle over $\caT(\sigma)$ 
\[
\prod_{\alpha\in\sigma}\,\spn\{\lambda\ssa\}\ \times\prod_{R^{\natural}\in\,
\operatorname{parts}^{\sharp}(\sigma)}T^{1,0}\caT(R^{\natural}),
\]
where $\operatorname{parts}^{\sharp}(\sigma)$ is the set of components that are not thrice-punctured spheres for the noded Riemann surfaces represented in $\caT(\sigma)$.  The asymptotic product structure for the metric first appeared in \cite{Msext} and appears in formula (\ref{wpherm}); for the covariant derivative the structure is detailed in \cite[Theorem 4.6]{Wlext} and Theorem \ref{curvcont} combined with the vanishing conventions displays the structure for the curvature tensor.   Evaluations involving more than a single factor tend to zero and evaluations for a single factor tend to evaluations for either the standard metric for opening a node or a lower dimensional Teichm\"{u}ller space.   The structure is formal since $\Tbar$ is not a complex manifold and the corresponding extension of the vector bundle of holomorphic quadratic differentials over $\Mbar$ is not the cotangent bundle, but the logarithmic polar cotangent bundle \cite{HMbook}. Nevertheless the product structure applies for limits of the metric and curvature tensor along $\caT(\sigma)$.  The individual product factors have strictly negative sectional curvature.   
The $\lambda\ssa$-section is holomorphic with curvature bounded above by Theorem \ref{wppastbd},  or equivalently bounded above by Corollary \ref{wpseccurv}.   The general result \cite{Trcurv,Wlchern} establishes negative curvature for the Teichm\"{u}ller spaces $\caT(R^{\natural})$.  Recall that for a product of negatively curved manifolds, a zero curvature tangent section has at most one $\mathbb R$-dimensional projection into the tangent space of each factor.  Accordingly, the maximal dimension of a flat tangent multi section equals the number of product factors. We establish the counterpart for WP.  We continue using pants decomposition gradients  as a $\mathbb C$-frame.  By considering $\mathbb C$-sums of the indeterminates $\lambda\ssg$, $\gamma\in\caP$, a germ $\mathcal V$ is defined for an extension over $\Tbar$ of the tangent bundle of $\caT$.  A formal product structure for $\mathcal V$ is defined 
\begin{equation}\label{formprod}
\prod_{\alpha\in\sigma}\,\spn\{\lambda\ssa\}\ \times\prod_{F_j\in\,\operatorname{parts}^{\sharp}(\sigma)}\spn\{\lambda\ssb\}_{\beta\in\caP,\,\beta\scriptsize{\mbox{ on }}F_j}
\end{equation}    
by considering the short geodesics $\sigma$ and the components $\{F_j\}$ of the complement of 
$\sigma$ in the base surface $F$; $\operatorname{parts}^{\sharp}(\sigma)$ is now the set of components not homeomorphic to a sphere minus three disjoint discs.  
\begin{corollary}\textup{Classification of asymptotic flats.}\label{flats}  Let $\mathcal S$ be a $\mathbb R$-subspace of the fiber of $\mathcal V$ over a point of $\caT(\sigma)$.  The subspace $\mathcal S$  is a limit of a sequence of tangent multisections over points of $\caT$ with all sectional curvatures tending to zero if and only if the projections of $\mathcal S$ onto the factors of  (\ref{formprod}) are at most one $\mathbb R$-dimensional.  The maximal dimension for $\mathcal S$ is $|\sigma|+|\operatorname{parts}^{\sharp}(\sigma)|\le\dim_{\mathbb C}\caT$.
\end{corollary}
\begin{proof}
We use the $\mathbb C$-frame $\lambda\ssg,\gamma\in\caP$, and the definition of $\mathcal V$.  Sectional curvature is given by formula (\ref{seccurv}).  Evaluation of the denominator of (\ref{seccurv}) is continuous and non zero by Lemma \ref{inprcts}.  As a consequence of the continuity of evaluations and vanishing of evaluations, approaching $\caT(\sigma)$, the contribution to the numerator of (\ref{seccurv}) tends to zero for the evaluations involving more than a single factor of the product (\ref{formprod}).  In particular, the numerator is given as separate sums for the factors of the product and a remainder which vanishes on $\caT(\sigma)$.   Since the individual factors have strictly negative sectional curvatures, it follows for a limit of sections with sectional curvatures tending to zero, that the curvature contribution for each factor tends to zero. It follows that the projection of a limit to each factor is one $\mathbb R$-dimensional.  Consider then the converse.  Since a frame is given, a subspace of a fiber determines a sub bundle.  A consequence of the continuity of evaluations and at most one  $\mathbb R$-dimensional projections is that the sub bundle sectional curvatures tend to zero approaching the limit fiber.  The dimension conclusion is immediate.   
\end{proof}

Understanding approximately flat subspaces is an important consideration for global geometry.   The rank in the sense of Gromov is the maximal dimension of a quasi-flat subspace, a quasi-isometric embedding of a Euclidean space. Brock-Farb find that the rank is one if $\dim_{\mathbb C}\caT\le 2$ and in general is at least the maximum $\mathbf m=\lfloor (1+\dim_{\mathbb C}\caT)/2\rfloor$ of $|\operatorname{parts}^{\sharp}(\sigma)|$ taken over simplices in the curve complex \cite{BF}.  Brock-Farb \cite{BF}, Behrstock \cite{Bhr} and Aramayona \cite{Arm} apply considerations to show that Teichm\"{u}ller space is Gromov hyperbolic in the rank one case.   Behrstock-Minsky find in their work on the asymptotic cone of the mapping class group that the rank is exactly $\mathbf m$ \cite{BMi}.  In \cite[Section 6]{Wlbhv} we find that a locally Euclidean subspace of $\Tbar$ has dimension at most the maximum $\mathbf m$.   The above corollary provides additional information on asymptotically flat subspaces and rank.  We apply compactness and the above corollary to show that beyond asymptotic flats there is a negative upper bound for sectional curvature.

\begin{corollary}\label{beyflat}
There exists a negative constant $c_{g,n}$ such that a subspace $\caS$ of a tangent space of $\caT$ with $\dim_{\mathbb R}\caS > \dim_{\mathbb C}\caT$ contains a section with sectional curvature at most $c_{g,n}$.  
\end{corollary}
\begin{proof}  Consider a tangent subspace $\mathcal S$ to $\caT$ and write $\msc(\mathcal S)$ for the minimal sectional curvature for sections of $\mathcal S$.  For the set $\{\mathcal S\}$ of all tangent subspaces to $\caT$ of a given dimension, let $\mathbf c=\sup_{\{\mathcal S\}}\msc(\mathcal S)$ be the supremum of minimal sectional curvatures.  The supremum is finite and non positive since the sectional curvatures of $\caT$ are negative.  We consider that $\mathbf c$ is zero.  Choose a sequence $\mathcal S_n$ of tangent subspaces with $\msc(\mathcal S_n)$ tending to zero.  We note that the mapping class group $Mod$ acts by isometries, the quotient $\Tbar/Mod$ is compact and the germs $\mathcal V$ provide an extension of the tangent bundle over $\Tbar$.  We can select a subsequence (same notation) of tangent subspaces and elements $\gamma_n\in Mod$ such that $\gamma_n\mathcal S_n$ converges to $\mathcal S'$, a subspace of a fiber of the extension $\mathcal V$.   The sectional curvatures of $\mathcal S'$ are zero.  Corollary \ref{flats} provides that $\dim_{\mathbb R}\mathcal S'\le \dim_{\mathbb C}\caT$.    The desired conclusion is the contrapositive statement.
\end{proof}    



\end{document}